\documentclass[11pt]{amsart}
\usepackage[utf8]{inputenc}
\usepackage{amsfonts, mathtools, nccmath, amssymb, graphicx, multicol, commath, bbm, mathrsfs, bm, subfig, mathrsfs}
\usepackage{amsthm, mathabx}
\usepackage{esint}
\usepackage{braket}
\usepackage[foot]{amsaddr}
\usepackage{bibentry}
\usepackage{color,dsfont}
\usepackage[top=1in, bottom=1in, left=1in, right=1in]{geometry}
\usepackage{hyperref}
\hypersetup{colorlinks,breaklinks,
             linkcolor=blue,urlcolor=blue,
             anchorcolor=blue,citecolor=blue}

\usepackage{cite}

\newcommand{\R}{\mathbb{R}}
\newcommand{\CP}{\mathbb{C}}
\newcommand{\N}{\mathbb{N}}
\newcommand{\Z}{\mathbb{Z}}
\renewcommand{\P}{\mathscr{P}}

\newcommand{\V}{\mathcal{V}}
\newcommand{\C}{\mathcal{C}}

\newcommand{\E}{\mathbb{E}}

\newcommand{\Leb}[1][]{\mathscr{L}}

\newcommand{\G}{\mathcal{G}}
\newcommand{\D}{\mathcal{D}}

\newcommand{\B}{\mathsf{B}}

\newcommand{\Ecal}{\mathcal{E}}

\newcommand{\sans}[1]{\mathsf{#1}}
\newcommand{\sP}{\sans{P}}

\newcommand{\sH}{\sans{H}}
\newcommand{\sG}{\sans{\Gamma}}
\newcommand{\sSig}{\sans{\Sigma}}
\newcommand{\sF}{\sans{F}}

\newcommand{\eps}{\varepsilon}

\newcommand{\id}{\bm{i}}

\DeclareMathOperator{\Cov}{Cov}

\DeclareMathOperator{\supp}{supp}

\DeclareMathOperator{\tr}{tr}

\DeclareMathOperator*{\argmin}{argmin}

\DeclareMathOperator{\Ent}{Ent}

\DeclareMathOperator{\kl}{KL}
\DeclareMathOperator{\law}{Law}

\DeclareMathOperator{\range}{Range}
\DeclareMathOperator{\rank}{Rank}
\DeclareMathOperator{\Ricc}{Ricc}
\DeclareMathOperator{\OT}{OT}
\DeclareMathOperator{\const}{const}

\newtheorem{theorem}{Theorem}[section]
\newtheorem{assumption}[theorem]{Assumption}

\newtheorem{lemma}[theorem]{Lemma}
\newtheorem{proposition}[theorem]{Proposition}
\newtheorem{corollary}[theorem]{Corollary}
\theoremstyle{definition}
\newtheorem{definition}[theorem]{Definition}

\newcommand{\nc}{\normalcolor}
\definecolor{byzantine}{rgb}{0.74, 0.2, 0.64}
\definecolor{darkgreen}{rgb}{0.1,0.6,0.1}
\definecolor{darkred}{rgb}{0.6,0,0}
\definecolor{lightgray}{rgb}{0.5,0.5,0.5}
\definecolor{darkorange}{rgb}{1.0, 0.55, 0.0}

\newcommand{\An}[1]{{\color{darkgreen}\textbf{#1}}}

 \newenvironment{listi}
  {\begin{list} 
 {(\roman{broj})}
{ \usecounter{broj}}
     \setlength{\labelwidth}{25pt}
  }
{   \end{list} }
\newcounter{broj}
\renewcommand{\hat}{\widehat}

\def\XXint#1#2#3{{\setbox0=\hbox{$#1{#2#3}{\int}$}
     \vcenter{\hbox{$#2#3$}}\kern-.5\wd0}}

\newtheorem{remarkx}[theorem]{Remark}
\newenvironment{remark}{\pushQED{\qed}\remarkx}{\popQED\endremarkx}

\numberwithin{equation}{section}

\title{Large-time behavior and accuracy of the Mean-Field Ensemble Kalman Filter in the Linear Detectable Setting}
\author{Franca Hoffmann, Sangmin Park, Andrew M. Stuart} 
\address{F. Hoffmann, S. Park, A.M. Stuart: Department of Computing and Mathematical Sciences, California Institute of Technology, Pasadena, CA 91125, USA}
\email{franca.hoffmann@caltech.edu, sangminp@caltech.edu, astuart@caltech.edu}

\begin{document}

\begin{abstract}
The ensemble Kalman filter (EnKF), originally developed in the geophysical sciences, is now widely used for state and parameter estimation problems in various domains of application, and in particular in high dimensional settings. 
It may be viewed as a robust, cheap-to-implement,
alternative to the optimal, Bayesian, filter. However, despite its empirical successes, theoretical understanding of its properties, in relation to the optimal filter, is in its infancy.
We contribute novel theoretical understanding, studying the behavior of the
EnKF in the infinite particle number setting; this is known as the
mean-field ensemble Kalman filter (MFEnKF). In the linear 
setting with Gaussian noise and Gaussian initial data the MFEnKF coincides
with the optimal filter, which is characterized by the Kalman filter. 
Our focus in this paper will be the linear setting with Gaussian noise, 
but with non-Gaussian initial data.

Under the assumptions of nondegeneracy of the signal noise and detectability of the signal-observation pair we make three specific contributions, studying both
discrete and continuous formulations of the methodology. First, in discrete time, we derive the MFEnKF via a variational approximation of the Bayes update using a covariance-weighted optimal transport metric. Secondly, we use this metric to show that the MFEnKF is a strict contraction towards the subspace of Gaussian measures, giving explicit geometric rates in terms of the covariances, in both discrete and continuous time. Finally, in the continuous-time setting, we deduce a stable form of almost sure accuracy of the MFEnKF: given any pair of initial laws (satisfying mild conditions) for the MFEnKF and the optimal filter, the 2-Wasserstein distance between their distributions decays exponentially, for almost every observation path.
\end{abstract}

\maketitle



\section{Introduction}\label{sec:intro}

The filtering problem, in discrete time, is to estimate the unknown state (or signal) $(X_n^\dag)$ of a Markov
process, given the observation sequence $(Y_j^\dag)_{1 \le j \le n}$, and to update this estimate sequentially 
in $n \ge 1$. In the Bayesian framework, the best estimate is the conditional law of the state at time $n$, given all observations up to time $n$. We refer to the conditional law as the optimal filtering distribution and the corresponding recursive process as the optimal filter, and we adopt a similar convention for filtering algorithms approximating the optimal filter. We also adopt similar conventions in continuous time.

In this paper we address a central issue arising when seeking algorithms to approximate the optimal filtering distribution. In principle any such algorithm will need to accurately solve a sequence of Bayesian inverse problems, given access to approximate priors, at each observation time; this presents an enormous computational burden when actioned over
long time horizons. However, in the presence of filter stability it may be possible to sidestep this heavy burden: it is natural to ask whether filter stability
can be used to correct for, or control, errors made in the initial stages of the algorithm. The paper is focused on
addressing this foundational question in the context of the widely deployed ensemble Kalman methodology.

Given a random vector $X_0^\dag\sim\mu^\dag$, consider the following Markov process for signal $\{X_n^\dag\}_{n\geq 0}$ in $\R^d$, and observation $\{Y_{n}^\dag\}_{n\geq 1}$ in $\R^k$ with $k\leq d$: for $n \in \Z_+$ we have
\begin{equation}\label{def:sig-obs;discrete}
\begin{split}
    X_{n+1}^\dag&= AX_{n}^\dag+\xi_{n+1}^\dag, \\
    Y_{n+1}^\dag&= HX_{n+1}^\dag+\eta_{n+1}^\dag,
\end{split}
\end{equation}
with matrices $A\in\R^{d\times d}, H\in\R^{k\times d}$, and vector-valued random sequences $\xi_n^\dag\overset{i.i.d.}{\sim} N(0,\Sigma)$ and $\eta_{n}^\dag\overset{i.i.d.}{\sim}N(0,\Gamma)$ independent from each other and from the initial data, and defined by positive definite covariance matrices
$\Sigma\in\R^{d\times d}$, $\Gamma\in\R^{k\times k}$. We refer to this setting as the \emph{linear setting with Gaussian noise}. If, in addition, the initial datum $X_0^\dag$ is also Gaussian, we refer to the setting as the \emph{linear-Gaussian setting}; in this case, the Markov process is Gaussian, and the optimal filtering distribution is characterized by the Kalman 
filter~\cite{Kalman60}. The Kalman filter is asymptotically stable,
with respect to differences in the  initial data, if the pair $(H,A)$ of 
observation and signal matrices satisfies a detectability condition~\cite{Wonham68riccati}. As we are primarily interested in large-time behavior of filtering distributions, we will often consider the \emph{linear detectable setting}, 
referring to the linear setting with Gaussian noise where, in addition,
the pair $(H,A)$ is assumed to be detectable, a concept we define precisely
within the paper.

The ensemble Kalman filter (EnKF), introduced by Evensen~\cite{Evensen94EnKF}, 
is an approximate filter which is computationally efficient in 
high-dimensional, nonlinear problems; it is a tractable alternative 
to the optimal filter when the latter is not itself tractable. In this paper, we study the infinite-particle limit of EnKF, namely the mean-field ensemble Kalman filter (MFEnKF) \cite{CalReiStu25}. We consider the linear setting with Gaussian noise defined by \eqref{def:sig-obs;discrete}. Let $\Cov(\pi)\in\R_+^{d\times d}$ denote the covariance of probability measure $\pi$.
Then the MFEnKF is defined by the following recursive process: given random vectors $X_0\sim\mu$, $\xi_n\overset{i.i.d.}{\sim} N(0,\Sigma)$, and $\eta_{n}\overset{i.i.d.}{\sim}N(0,\Gamma)$ independent from each other and $X_0$, define, for $n \in \Z_+$,
\begin{equation}\label{def:MFEnKF;disc}
\begin{split}
    \hat X_{n+1} &= AX_n+\xi_{n+1},\\
    X_{n+1} &=\hat X_{n+1}-\Cov(\hat \pi_{n+1}^\mu) H^T(H\Cov(\hat \pi_{n+1}^\mu) H^T+\Gamma)^{-1}(H\hat X_{n+1}+\eta_{n+1}-Y_{n+1}^\dag);
\end{split}
\end{equation}
here $\hat\pi_{n+1}^\mu$, which we refer to as the \emph{predictive MFEnKF distribution}, is the conditional law of $\hat X_{n+1}$ given observations $(Y_j^\dag)_{1\leq j\leq n}$, and we note its dependence on the initial law $\mu$ of $X_0$ by the superscript. Similarly, we refer to the conditional law $\pi_{n+1}^\mu$ of $X_{n+1}$ given $(Y_j^\dag)_{1\leq j\leq n+1}$ as the \emph{MFEnKF distribution}. The key property of the MFEnKF in this linear setting with Gaussian noise is that, if $\mu$ is Gaussian, then the MFEnKF distribution given via \eqref{def:MFEnKF;disc} is Gaussian at all times and is equal to the optimal filtering distribution initialized at $\mu$ given via \eqref{def:sig-obs;discrete} -- the
Kalman filter.
In this paper we study the setting where $\mu$ is not Gaussian, and the resulting relation of the MFEnKF to the optimal filter. 
We make the following contributions.
\begin{listi}
    \item[{\bf (C1)}] The optimal filter is defined by an iteration
which interleaves a prediction step with solution of a Bayesian inverse 
problem to incorporate each new data point. We recover the MFEnKF by
approximating the traditional variational formulation of the 
inverse problem, in the linear setting with Gaussian
noise. We employ a variational approximation based on a 
covariance-weighted optimal transport metric; this gives rise to
a variational interpretation of the mean-field ensemble Kalman methodology
(Section~\ref{sec:EnKF;var_der}).
    \item[{\bf (C2)}] We show that the (predictive) MFEnKF distributions contract exponentially towards the Gaussian manifold with respect to a covariance-weighted optimal transport metric; see Theorem~\ref{thm:gauss_lyapunov;disc} for discrete time and Theorem~\ref{thm:gauss_lyapunov;cts} for continuous time models. Furthermore, the contraction rates are explicitly characterized in terms of the covariance used in the weighted metric.
    \item[{\bf (C3)}] As a consequence, we deduce the large-time stability (Theorem~\ref{thm:MFEnKF;stable;cts}) and accuracy in relation to the optimal filter (Corollary~\ref{cor:MFEnKF_accurate;cts}) of the continuous-time MFEnKF, with explicit asymptotic exponential rates in the 2-Wasserstein distance.
\end{listi}

Contribution {\bf (C1)} concerns only approximation of the
solution of arbitrary linear inverse problems with additive Gaussian noise, 
for arbitrary, including non-Gaussian, prior distributions; 
it may therefore be of independent interest. In the context of this paper, 
the formulation leads to identification of the correct metric in which 
to quantify stability of the MFEnKF. Indeed,
Contribution {\bf (C2)} uses this metric to identify a contraction structure,
analogous to the time-varying Lyapunov function used in classical stability analysis 
of the Kalman filter~\cite{anderson71stability}, for the MFEnKF.
Contractivity will, of course, be
useful to control multiple forms of error propagation related to the MFEnKF.
Contribution {\bf (C3)} asserts that the MFEnKF satisfies an ideal 
property: the MFEnKF distribution forgets the initialization exponentially fast, and 
hence recovers the distribution of the optimal filter distribution regardless of the initial laws of both filters, under mild conditions.
This asymptotic filter accuracy is established under conditions comparable to those that guarantee asymptotic stability of the Kalman filter, namely detectability. Henceforth, we will simply refer to this property as accuracy; this concept should not be confused with the commonly considered notion of accuracy in the sense of state estimation.

The paper is organized as follows. We summarize our notation in Section~\ref{ssec:notation}
and make our setting precise in Section~\ref{ssec:setting}. Section~\ref{ssec:summary} is devoted 
to a summary of our main results, and is followed by a literature review in Section~\ref{ssec:literature}.
In Section~\ref{sec:EnKF;var_der} we introduce and deploy the covariance-weighted metric to yield results that constitute {\bf (C1)}; the connection between discrete and continuous-time MFEnKF is discussed in Section~\ref{ssec:disc-cts}. 
Section~\ref{sec:linear} is devoted to the study of the dynamics of the MFEnKF, obtaining quantitative contraction toward the Gaussian manifold {\bf (C2)} in both discrete and continuous time; we then build on this to
deduce stability and accuracy {\bf (C3)} in continuous time.

\subsection{Notation}\label{ssec:notation} 
We write $\R$, $\R_+$, $\R_{++}$ for the reals, nonnegative reals, and strictly positive reals, respectively; $\CP$ denotes the complex plane.
We write $\N$ for the positive integers and $\Z_+:=\N \cup \{0\}.$ For $q\in\mathbb{N}$, $\R^q$ is the $q$-dimensional Euclidean space, and $\Leb^q$ denotes the $q$-dimensional Lebesgue measure. For $l\in\Z_+$, the space of $l$-times continuously differentiable functions (resp. space of functions with up to $l$ continuous and bounded derivatives) from $\R^q$ to $\R^p$,  given $p,q\in\N$, is denoted by $\C^l(\R^q;\R^p)$ (resp. $\C_b^l(\R^q;\R^p)$), where the superscript $l$ is omitted if $l=0$; if $p=1$ we simply write $\C^l(\R^q)$ (resp. $\C_b^l(\R^q)$). Given $f\in\C^1(\R^q)$, its gradient is denoted by $\nabla f$, and its Hessian by $D^2 f$ if $f\in\C^2(\R^q)$.

We write $\R^{p\times q}$ for real $p$ by $q$ matrices, and $\R^{q\times q}_{sym}$ for real $q$ by $q$ symmetric matrices. Given $M\in\R^{q\times q}_{sym}$, we write $M\succeq 0$ and $M\succ 0$ for, respectively, positive semi-definite and positive definite $M$. We use $\R^{q\times q}_{+},\R^{q\times q}_{++}$ to denote the cones of positive semi-definite and positive definite matrices, respectively. Given $M\in\R^{q\times q}_{++}$, we write the weighted inner-product on $\R^q$ as $\langle v, w \rangle _{M}:=v^T M^{-1} w$, and $|\cdot|_{M}$ to denote the induced norm:
\begin{equation}\label{def:Mweightednorm}
    |v|_M^2 := v^T M^{-1} v \text{ for each } v\in\R^q.
\end{equation}
When $M=I$ we obtain the standard Euclidean norm and simply write $|\cdot|$, noting that the same notation is also used
for complex modulus.
The identity matrix on $\R^q$ is denoted $I_q\in\R^{q\times q}_{++}$; when $q$ is clear from context, we will often omit the subscript and write $I=I_q$.

Given $M\in\R^{q\times q}$, we define the spectral radius
\begin{equation}\label{def:rho}
    \rho(M):=\max\{|\lambda|:\;\lambda\text{ is an eigenvalue of } M\}
\end{equation}
and spectral abscissa
\begin{equation}\label{def:alpha}
    \alpha(M):=\max\{\Re(\lambda):\;\lambda\text{ is an eigenvalue of } M\}
\end{equation}
where $\Re:\CP \to \R$ denotes the map of a complex number to its real part.
Given a symmetric matrix $S\in\R^{q\times q}$, we write $\lambda_{\max}(S)=\alpha(S)$ and $\lambda_{\min}(S)=-\alpha(-S)$, the maximum and minimum eigenvalues of $S$.

The space of probability measures on $\R^q$ is denoted by $\P(\R^q)$, and $\P_2(\R^q)$ the subset
of probability measures with finite second moments. We write $X\sim\mu$ for $X$ distributed according to
$\mu\in\P(\R^q)$. 

Given $\mu\in\P_2(\R^q)$ and any Borel measurable function $\varphi:\R^q\rightarrow\R$, growing at most quadratically at infinity, we write $\mu(\varphi)=\int_{\R^q}\varphi(u)\;\mu(du)$. For $\boldsymbol{\varphi}=(\varphi^1,\cdots,\varphi^q)$ 
we write $\mu(\boldsymbol{\varphi})=(\mu(\varphi^1),\cdots,\mu(\varphi^q))\in\R^q$. 
 The identity function on $\R^q$ is written $\id_q$. We write $\id_q^1,\id_q^2:\R^q\times\R^q\rightarrow\R^q$ to refer to the projection onto the first and second components. When $q$ is clear from context, we simply write $\id=\id_q$, $\id^1=\id_q^1$, and $\id^2=\id_q^2$.

We write $\Cov(\mu)\in\R^{q\times q}_+$ to refer to the covariance matrix of $\mu$. Given a random vector $X\sim \mu$, we overload the notation and write $\Cov(X):=\Cov(\mu)$. We reserve the upper case letter $C$ for covariance matrices, which play a central role throughout this paper.
We will sometimes write $\P_2^g(\R^q)$ to denote the subspace of $\P_2(\R^q)$ consisting of all Gaussian measures, including those (such as Dirac measures) with singular covariance matrices. Given $m\in\R^q,C\in\R^{q\times q}_+$, we write $N(m,C)$ to denote the Gaussian distribution with mean $m$ and covariance 
$C \succeq 0$. We define the Gaussian projection $\G:\P_2(\R^q)\rightarrow \P_2^{g}(\R^q)$ mapping $\mu$ to $N(\mu(\id),\Cov(\mu))$, the normal distribution with mean and covariance of $\mu$. We note that when $\mu$ is sufficiently regular, the Gaussian projection admits the following variational interpretation in terms of
the Kullback-Leibler divergence:
\begin{equation}\label{def:G}
    \G\mu=\argmin_{\sigma\in\P_2^g(\R^q)}\kl(\mu||\sigma).
\end{equation}

The matrix-weighted Wasserstein distances which we introduce in Section~\ref{ssec:W2M} will be a central tool throughout the paper. Given $M\in\R^{q\times q}_{++}$, we denote by $W_{2}(\cdot,\cdot;M):\P_2(\R^q)\times\P_2(\R^q)\rightarrow \R_+$ the $M$-weighted 2-Wasserstein distance; often we will abbreviate and write $W_{2,M}(\cdot,\cdot)$.

We write $\sP:\P_2(\R^d)\rightarrow\P_2(\R^d)$ to refer to the map corresponding to the signal dynamics in \eqref{def:sig-obs;discrete}, $X\mapsto AX+\xi$ for $\xi\in N(0,\Sigma)$, that is,
\begin{equation}\label{def:P}
    \sP\pi(dx)=\frac{1}{\sqrt{(2\pi)^d\det\Sigma}}\int_{\R^d}\exp\left(-\frac{|x-Ax'|^2_{\Sigma}}{2}\right)\pi(dx')\;dx.
\end{equation}
Regarding the observation dynamics, we write $\B:\P_2(\R^d)\times\R^k\rightarrow\P_2(\R^d)$ to denote the prior to posterior map implied by use of Bayes theorem to incorporate data $y$ from the inverse problem 
\begin{equation}\label{eq:LIP}
    y=Hx+\eta
\end{equation}
with prior $\pi:$
\begin{equation}\label{def:L}
    \B(\pi;y)(dx)=\frac{\exp\left(-\frac12|y-Hx|_\Gamma^2\right)\pi(dx)}{\int_{\R^d}\exp\left(-\frac12|y-Hx'|_{\Gamma}^2\right)\;\pi(dx')}.
\end{equation}

\subsection{Setting}\label{ssec:setting} 
We now describe our setting precisely, starting with definitions of
the optimal filtering and MFEnKF recursions in Section~\ref{sssec:measures}.
This is followed, in Section~\ref{sssec:KF}, by definition of the
Kalman filter and the attendant Riccati equation. 
Continuous-time formulations are introduced in Section~\ref{sssec:cts}.

\subsubsection{Filtering Distributions}\label{sssec:measures}
Let $(X_n^\dag,Y_n^\dag)_{n\geq 0}$ be the signal-observation pair from~\eqref{def:sig-obs;discrete}; we use superscript $\dag$ for the random vectors appearing in the underlying signal-observation model to distinguish them from their counterparts in the MFEnKF~\eqref{def:MFEnKF;disc}. 
Given $\mu\in\P(\R^d)$, the optimal filtering distribution $\bar\pi^{\mu}_n$ is the conditional law  of $X_n^\dag$ given $(Y_j^\dag)_{j\leq n}$ and $X_0^\dag\sim\mu$. More precisely, $\bar\pi_0^{\mu}=\mu$ and
\begin{equation}\label{def:barpi_n}
    \bar\pi_n^{\mu}(\varphi):=\E[\varphi(X_n^\dag)|(Y_j^\dag)_{j\leq n}] \text{ for each } \varphi\in\C_b(\R^d).
\end{equation}
The sequence $(\bar\pi_n^{\mu})_{n\geq 0}$ can be characterized by the following recursive process (which we refer to as the optimal filter): 
\begin{equation}\label{eq:barpi:algo}
    \bar\pi_{n+1}^{\mu}=\B(\sP\bar\pi_{n}^{\mu};Y_{n+1}^{\dag})\text{ for } n\geq 0, \qquad \bar\pi_0^{\mu}=\mu.
\end{equation}
In the filtering and data assimilation literature, it is common to refer to the first step $\pi\mapsto \sP\pi$ as the \emph{forecast step}, and the Bayesian update as the \emph{analysis step}, language we adopt in this paper.

The MFEnKF distribution, the conditional law $\pi_n^\mu$ of $X_n$ 
satisfying \eqref{def:MFEnKF;disc}, given observations $(Y_j^\dag)_{j\leq n}$ and the initial law $X_0\sim\mu$, can be characterized by the recursion
\begin{equation}\label{eq:MFEnKF;algo}
\begin{split}
    \pi_{n+1}^\mu&=\B^{EK}(\mathsf{P}\pi_n^\mu;Y_{n+1}^\dag) \text{ for } n\geq 0, \qquad\pi_0^\mu=\mu;
\end{split}
\end{equation}
here $\B^{EK}:\P_2(\R^d)\times \R^k\rightarrow\P_2(\R^d)$ 
is defined by
\begin{equation}\label{eq:hatB}
\begin{split}
    \B^{EK}(\mu;y)(\varphi)&=\E^{\eta\sim N(0,\Gamma), X\sim\mu}\left[\varphi\left(T^{EK}(X;\mu,y-\eta)\right)\right] \text{ for each } \varphi\in\C_b(\R^d)\\
    T^{EK}(x;\mu,v)&:=x-\Cov(\mu) H^T(H\Cov(\mu)H^T+\Gamma)^{-1}(Hx -v);
\end{split}
\end{equation}
note that the analysis step of the MFEnKF~\eqref{def:MFEnKF;disc} may be written
\[X_{n+1}=T^{EK}(\hat X_{n+1};\hat\pi_{n+1}^\mu,Y_{n+1}^\dag-\eta_{n+1}).\]
where $\hat\pi_{n+1}^\mu$ is the conditional law of $\hat X_{n+1}$ in \eqref{def:MFEnKF;disc} given $(Y_j^\dag)_{j\leq n}$.
Unlike the Bayes update $\B$, the operation $\B^{EK}$~\eqref{eq:hatB} is \emph{Lagrangian}, in the sense that it provides a recipe to move each particle in the support of the prior $\hat\pi_{n+1}^\mu$, and therefore provides a suitable basis for a tractable, particle-based algorithm. Indeed, the only macroscopic information the update requires from the prior is its covariance, which is easily computed at the level of samples.
Comparing \eqref{eq:barpi:algo} and \eqref{eq:MFEnKF;algo} we see that it
is natural to try and understand the sense in which $\B^{EK}$ approximates 
$\B.$

Note that both $\bar\pi_n^{\mu}$ and $\pi_n^\mu$ depend on the observation sequence $(Y_j^\dag)_{j\leq n}$, hence are random probability measures. In discrete-time, we will often treat the realization of the observation as a fixed deterministic sequence; in this case, we will use the lower case letter and write $(y_j^\dag)_{j\leq n}$. This will help simplify our notation -- for instance, $\hat\pi_{n+1}^\mu$ in \eqref{def:MFEnKF;disc} is simply $\law(\hat X_{n+1})$ -- without changing the nature of our results, which can be understood to hold for every realization $Y_n^\dag=y_n^\dag$.

For simplicity of notation, we will sometimes omit the superscript $\mu$ for the MFEnKF distribution or the optimal filtering distribution when the dependence on the initial law is not important, for instance in Section~\ref{sec:EnKF;var_der} where the main focus is on the recursive algorithm, not the behavior of the distributions in terms of their dependence on the initial law. Conversely, we will sometimes use the superscript $\mu$ for the random vectors $X_n^\mu,\hat X_n^\mu$ or the associated noise $\xi_n^\mu,\eta_n^\mu$ when we are comparing different processes with different initial law, and the dependence is relevant.

\subsubsection{Kalman Filter and Riccati Equation}\label{sssec:KF}
In the linear-Gaussian setting, when the dynamics takes the form \eqref{def:sig-obs;discrete} and the initial datum is a Gaussian measure, Gaussianity propagates and the optimal filtering distribution is fully characterized by its mean and covariance -- the Kalman filter~\cite{Kalman60}. Thus, the problem becomes finite-dimensional. Furthermore, in the linear-Gaussian setting, the MFEnKF distribution coincides exactly with this optimal, Kalman, filtering distribution. The Kalman filter, governing the
evolution of both $\bar\pi_{n}^{\mu}$ and $\pi_{n}^{\mu}$ in the linear-Gaussian
setting, is defined by the following evolution equations for $(m_n,C_n)_{n\geq 1}$: 
\begin{equation}\label{def:KF;disc}
\begin{split}
    \hat m_{n+1} &= A m_n,\\
    \hat C_{n+1} &= A C_n A^T+\Sigma,\\
    m_{n+1} &= \hat m_{n+1}-\hat C_{n+1} H^T(H \hat C_{n+1}H^T+\Gamma)^{-1}(H\hat m_{n+1}-y_{n+1}^\dag),\\
    C_{n+1} &= \hat C_{n+1} - \hat C_{n+1} H^T(H\hat C_{n+1}H^T+\Gamma)^{-1}H\hat C_{n+1};
\end{split}
\end{equation}
$m_0\in\R^d$ and $C_0\in\R_{+}^{d\times d}$ are the mean and covariance of the Gaussian initial datum $\mu\in\P_2(\R^d)$.

The sequence of covariance matrices $(\hat C_n)_{n\geq 1}$ satisfies the recursion
\begin{equation}\label{eq:Riccd;recursion}
\hat C_{n+1}=\Ricc_{d}(\hat C_n),
\end{equation}
where the discrete Riccati operator $\Ricc_{d}:\R^{d\times d}_+\rightarrow\R^{d\times d}_{sym}$ is defined by
\begin{equation}\label{def:Riccd}
\Ricc_{d}(C):=A C A^T+\Sigma-A C H^T(H C H^T+\Gamma)^{-1}H C A^T.
\end{equation}
Under detectability of $(H,A)$ and controllability (or stabilizability) of $(A,\Sigma)$ (see Definition~\ref{def:detectability}), $\hat C_n$ converges to the 
unique fixed point of $\Ricc_{d}$ 
in the cone of positive definite matrices.

The crucial fact we use throughout this paper to understand the dynamics of the MFEnKF in the linear setting with Gaussian noise is that the covariance matrices of the MFEnKF satisfy the Riccati recursion~\eqref{eq:Riccd;recursion}, \emph{regardless of whether the initial datum is Gaussian}; see Propositions~\ref{prop:Riccati-d}-\ref{prop:Riccati-c}. This property has been previously noted in the literature, for instance by Del Moral and Tugaut~\cite{DelMoralTugaut18stability}, and is not shared by the optimal filter itself.

\subsubsection{Continuous-time Model}\label{sssec:cts}
We now consider continuous-time analogs of the
discrete-time setting considered until this point. To this
end we introduce the time-step $\tau>0$ and write
\begin{equation}\label{def:tau_scale}
A=I+\tau \sF,\quad H=\tau\sH, \quad \Gamma=\tau\sG, \qquad \Sigma=\tau\sSig.
\end{equation}
Here $\sF\in\R^{d\times d}$, $\sSig\in\R^{d\times d}_{++}$, $\sH\in\R^{k\times d}$, $\sG\in\R^{k\times k}_{++}$. Changing variables $Y_{n+1}^\dag=Z_{n+1}^\dag-Z_n^\dag$ with $Z_0^\dag=0$, and taking the limit $\tau\to 0$, we obtain the signal and observation processes
\begin{equation}\label{eq:linearSigObs;cts}
\begin{split}
    dX_t^\dag &= \sF X_t^\dag\;dt +\sqrt{\sans{\Sigma}}\;dB_t^\dag,\\
    dZ_t^\dag &= \sH X_t^\dag\;dt + \sqrt{\sG}\;dW_t^\dag;
\end{split}
\end{equation}
here, $(B_t^\dag)_{t\geq 0},(W_t^\dag)_{t\geq 0}$ are standard Brownian motions independent from each other and the initial datum. See, for instance, \cite[Section 3.1]{CalReiStu25} for a more detailed derivation.

The continuous-time mean-field Ensemble Kalman filter is
\begin{equation}\label{eq:EnKF;mf;cts}
    d X_t=\sF X_t\;dt + \sqrt{\sans{\Sigma}}\;dB_t - \Cov(\pi_t^\mu)\sH^T\sG^{-1}(\sH X_t\;dt+\sqrt{\sG}dW_t- dZ_t^\dag),\quad X_0\sim\mu
\end{equation}
where $\pi_t^\mu$ is the MFEnKF distribution in continuous time, namely the conditional law of $X_t$ given the observations $(Z_s^\dag)_{s\leq t}$, and we use the superscript to note the dependence on the initial law $\mu$ of $X_0$. Indeed, one can derive~\eqref{eq:EnKF;mf;cts} as the limit of the discrete-time MFEnKF~\eqref{def:MFEnKF;disc} as $\tau\to 0$ under the scaling \eqref{def:tau_scale}; see Section~\ref{ssec:disc-cts} for details.

In continuous-time, we will always treat the observation process $(Z_t^\dag)_{t\geq 0}$ as random, firstly to avoid having to deal with rough deterministic observation paths, and secondly to compare the MFEnKF distributions with the optimal filtering distribution, study of which relies on the randomness of the observations.

The covariance matrix $C_t:=\Cov(\pi_t^\mu)$ satisfies the Riccati differential equation, 
\begin{equation}\label{eq:Riccc;diff}
\dot C_t=\Ricc_{c}(C_t)
\end{equation}
with $C_0=\Cov(\mu)$, where the continuous Riccati operator $\Ricc_{c}:\R^{d\times d}_+\rightarrow\R^{d\times d}_+$ is defined by
\begin{equation}\label{def:Riccc}
    \Ricc_{c}(C):=C\sF^T+\sF C+\sSig-C\sH^T\sG^{-1}\sH C.
\end{equation}
As in discrete time, the covariance evolution is independent of the observation $(Z_t^\dag)_{t\geq 0}$ and does not require Gaussianity of the initial distribution.
Equation \eqref{eq:Riccc;diff} can be derived from \eqref{eq:Riccd;recursion}
under the scaling \eqref{def:tau_scale}.

\subsection{Summary of Main Results}\label{ssec:summary}

The optimal filter~\eqref{eq:barpi:algo} and the MFEnKF~\eqref{eq:MFEnKF;algo} both depend on the observation sequence and the initial distribution. In typical applications, only one realization of observation sequence is available, and we are interested in approximating the optimal filter distribution without access to the true initial distribution. For this
reason we are interested in the stability of the filtering distributions with respect to initial data. Our notation
therefore makes explicit the dependence on the initial measure in both $\bar\pi_{n}^{\mu}$ and $\pi_{n}^{\mu}$ (and in their
continuous time analogs), but does not highlight dependence on the observation sequence. Indeed, whenever we compare filtering distributions, from the optimal filter or the MFEnKF, we always assume that they are driven by the same observations. Moreover, we will take a \emph{pathwise} viewpoint in this paper; namely, we are interested in the behavior of the filtering distributions for each (almost every) sample path of $(Y_n^\dag)_{n\geq 1}$. In Sections~\ref{sssec:C1}-\ref{sssec:C3} we outline the significance of, and methodology leading to, our three
main contributions.

\subsubsection{Contribution {\bf (C1)}}\label{sssec:C1}
It is well known that (see~\cite{ba2022randomized}, for instance), in the linear setting with Gaussian noise, the output $X_n$ of the analysis step of the MFEnKF~\eqref{def:MFEnKF;disc} can be written in terms of minimization of a randomized quadratic objective function:
\begin{equation}\label{eq:MFEnKF;tikhonov}
    X_{n}\in\argmin_{x}\Bigl(\frac12|x-\hat X_{n}|_{\Cov(\hat X_{n})}^2 + \frac12 |H x + \eta_{n}-y_{n}^\dag|_{\Gamma}^2\Bigr).
\end{equation}
Indeed, update rule \eqref{def:MFEnKF;disc} can be derived 
by application of the Woodbury matrix identity 
to the Euler-Lagrange equation resulting from optimizing
\eqref{eq:MFEnKF;tikhonov}; see also proof of Theorem~\ref{thm:MFEnKF;var}.
Contribution {\bf (C1)} provides a variational formulation of the law $\pi_{n}^{\mu}$ of $X_n$, the output of
the analysis step of the MFEnKF. Amongst multiple potential uses in the study of the algorithm, this perspective 
clarifies the origin of the weights $\Cov(\hat X_{n+1})$ and $\Gamma$ in terms of the underlying Bayes update.

Consider the Bayesian inverse problem \eqref{eq:LIP} with
prior $\pi_{prior}$. The posterior measure $\bar\pi_{post}=\B(\pi_{prior};y)$ with the Bayes operator $\B$ defined as in \eqref{def:L} can be equivalently characterized by the variational formulation of Bayes theorem 
\begin{equation}\label{eq:bayes;variational}
\bar\pi_{post}\in\argmin_{\mu\in\P_2(\R^d)}\Bigl(\kl(\mu||\pi_{prior})+\frac{1}{2}\int_{\R^d}|Hx-y|^2_\Gamma\;d\mu(x)\Bigr).
\end{equation} 
Although explicitly solvable for Gaussian $\pi_{prior}$,
this optimization formulation of Bayesian inversion requires, in general, some form of approximation
\cite{wainwright2008graphical}. In the context of filtering, a computationally desirable approximation would provide an 
explicit transport map, moving `each particle in' $\pi_{prior}$ to obtain (approximately) a corresponding particle in $\pi_{post}$.
With this goal in mind, it is natural to approximate \eqref{eq:bayes;variational} by replacing the Kullback-Leibler divergence with an optimal transport metric, as then the associated Euler-Lagrange equation specifies a transport map, under suitable assumptions. 

A natural problem, then, is to determine a suitable cost for the optimal 
transport distance to use in such an approximation. In Section~\ref{ssec:W2M} we introduce and study basic properties of the matrix-weighted 2-Wasserstein distances $W_2(\cdot,\cdot;M)$ (often abbreviated as $W_{2,M}$), which is a modification of the 2-Wasserstein distance, with the cost function $(x,y)\mapsto |x-y|^2$ replaced with $(x,y)\mapsto |x-y|_M^2$, given a positive definite matrix $M\in\R^{d\times d}_{++}$~\cite{reich2015probabilistic}.
For each probability measure $\pi(dx)\propto \exp(- U(x))\;dx$, we would like to know what choice of $M$ will lead to a good approximation of $\kl(\cdot||\pi)$. Proposition~\ref{prop:funct_ineq;W2M} establishes that, if $D^2 U(x)\geq M^{-1}\succ 0$, then
\[\frac12 W_2^2(\mu,\pi;M)\leq \kl(\mu||\pi)\leq \frac12 W_2(\mu,\pi;M)\left(2\sqrt{I_M(\mu||\pi)}-W_2(\mu,\pi;M)\right);\]
the left and right inequalities are respectively the Talagrand and the HWI-type inequalities, and $I_M$ is the $M$-weighted Fisher information (see~\eqref{def:IM}). From the above bounds and the discussion around the
Cramer-Rao and Brascamp-Lieb inequalities stated in~\eqref{eq:BL+CR;hatC}, 
it will emerge that a sensible choice for covariance-weighting is $M=\Cov(\pi)$.

Consequently, in Section~\ref{sec:EnKF;var_der} we arrive at the variational problem that is a modification of \eqref{eq:bayes;variational} to approximate the posterior $\bar\pi_{post}$, namely
\begin{equation}\label{eq:MFEnKF;var;summary}
    \pi_{post,\eta}\in\argmin_{\mu\in\P_2(\R^d)}\Bigl(\frac12 W_{2}^2(\mu,\pi_{prior};\Cov(\pi_{prior}))+\frac{1}{2}\int_{\R^d}|Hx+\eta-y|^2_\Gamma\;d\mu(x)\Bigr).
\end{equation}
Note that, in addition to replacing $\kl$ by the covariance-weighted Wasserstein distance, we have introduced observational noise $\eta\sim N(0,\Gamma)$ into
the data fidelity term. This is analogous to the presence of $\eta_n$ in \eqref{eq:MFEnKF;tikhonov} and 
arises because of the Lagrangian and implicit nature of the variational problem; see Section~\ref{sec:EnKF;var_der} for a more detailed discussion. Thus $\pi_{post,\eta}$ is a random measure, but once the uniqueness of the minimizer is verified, we can take expectation over $\eta$ to obtain a deterministic approximation of the posterior  -- i.e. $\pi_{post}(\varphi):=\E^{\eta}[\pi_{post,\eta}(\varphi)] \text{ for each } \varphi\in\C_b(\R^d)$. 

In Theorem~\ref{thm:MFEnKF;var} we show that the variational problem~\eqref{eq:MFEnKF;var;summary} has a unique minimizer, and coincides with the analysis step of the mean-field ensemble Kalman filter~\eqref{eq:MFEnKF;tikhonov} under the following correspondence:  $\eta=\eta_n$, $y= y_n^\dag$, and $\pi_{prior}=\hat\pi_n$, the conditional law of $\hat X_n$ given the observations up to time $n-1$.
Furthermore, if we are in the setting of the Kalman filter,
$\pi_{prior}$ is Gaussian and the approximate posterior $\pi_{post}$ 
coincides with the true posterior $\bar\pi_{post}$ 
from \eqref{eq:bayes;variational}.

\subsubsection{Contribution {\bf (C2)}}\label{sssec:C2}

In Section~\ref{sec:linear} we turn to analyzing the behavior of the MFEnKF with non-Gaussian initial data. We consider the linear detectable setting, both in discrete and continuous time, which in particular guarantees asymptotic stability of the Kalman filter; see Section~\ref{ssec:prelim} for further details. 

We show in Theorem~\ref{thm:gauss_lyapunov;disc} the discrete time version of Contribution {\bf (C2)}, namely that
\begin{equation}\label{sum:gauss_conv;disc}
W_2^2(\hat\pi_{n+1}^\mu,\hat\pi_{n+1}^{\G\mu};\Cov(\hat\pi_{n+1}^\mu))\leq (1-\beta_n) W_2^2(\hat\pi_n^\mu,\hat\pi_n^{\G\mu};\Cov(\hat\pi_n^\mu))
\end{equation}
for some explicitly characterized constants $\beta_n\in (0,1)$ each depending on $\Cov(\hat\pi_n^\mu)$ (and thus depending on $\mu$ only through $\Cov(\mu)$) and bounded away from zero uniformly in $n$. 

The difficulty in obtaining strictly contractive bounds at each step as in \eqref{sum:gauss_conv;disc} lies in the fact that the detectability of $(H,A)$ and controllability of $(A,\Sigma)$ cannot rule out initial transient behavior of the MFEnKF when measured in the Euclidean norm. More precisely, the deterministic part of the evolution of the predictive process $(\hat X_n)_n$ in~\eqref{def:MFEnKF;disc} is governed by the transition matrix
\[\Phi_n:=A-A\Cov(\hat\pi_n^\mu) H^T(H \Cov(\hat\pi_n^\mu) H^T+\Gamma)^{-1}H\in\R^{d\times d}.\]
Given $d$-detectability and controllability (see Definition~\ref{def:detectability}), $\Phi_n$ converges to $\Phi_\infty$ with $\rho(\Phi_\infty)<1$. However, this does not guarantee $|\Phi_\infty(x-y)|<|x-y|$, which requires a stronger condition $\rho(\Phi_\infty)\leq \norm{\Phi_\infty}_2<1$.
Contribution {\bf (C2)} exploits the fact that $\Phi_n^T \Cov(\hat\pi_{n+1}^\mu)^{-1}\Phi_n\preceq(1-\beta_n)\Cov(\hat \pi_n^\mu)^{-1}$, which is somewhat surprising given that $\rho(\Phi_n)$ may be much larger than $1$ for small $n$.

In the continuous-time setting, we assume c-detectability of $(\sH,\sF)$ (see Definition~\ref{def:detectability}) and $\sSig,\sG\succ 0$ and establish the counterpart to \eqref{sum:gauss_conv;disc} in Theorem~\ref{thm:gauss_lyapunov;cts}: that is, given $\mu\in\P_2(\R^d)$ we characterize $\beta_t>0$, depending on $\Cov(\pi_s^\mu)$ for $s\leq t$ and uniformly bounded away from $0$ for all $t\geq 0$, such that for each sample path of the observation, the continuous-time MFEnKF distributions $\pi_t^\mu,\pi_t^{\G\mu}$ initialized at $\mu,\G\mu$ (see \eqref{eq:EnKF;mf;cts}) satisfy
\begin{equation}\label{sum:gauss_conv;cts}
W_2(\pi_t^\mu,\pi_t^{\G\mu},\Cov(\pi_t^\mu))\leq\exp\left(-\frac12\beta_t t\right)W_2(\mu,\G\mu;\Cov(\mu)).
\end{equation}

The estimates \eqref{sum:gauss_conv;disc}-\eqref{sum:gauss_conv;cts} suggest that the covariance-weighted Wasserstein distance provides the right lens to capture the off-Gaussian behavior of the MFEnKF. Although this distance
depends on the time index through the covariance matrices, they stabilize rapidly under the detectability condition; the Riccati update is a strict contraction in suitable metrics~\cite{Bougerol93}. These distances have close connections with the time-varying Lyapunov functions introduced in the classical work of Anderson~\cite{anderson71stability} on the stability of the Kalman filter.

\subsubsection{Contribution {\bf (C3)}}\label{sssec:C3}
Recall that, on the Gaussian manifold, both the optimal filtering and MFEnKF distributions coincide with the Kalman filtering distribution. Thus, in \eqref{sum:gauss_conv;cts} $\pi_t^{\G\mu}=\bar\pi_t^{\G\mu}$, which connects
the MFEnKF directly to the optimal filter, and allows us to deduce stability and accuracy of the MFEnKF with non-Gaussian initial data.

In Theorem~\ref{thm:MFEnKF;stable;cts}, we establish the asymptotic stability of the MFEnKF in the usual 2-Wasserstein distance, a simplified statement of which is as follows: for any $\mu,\nu\in\P_2(\R^d)$ with nondegenerate covariance,
\[\limsup_{t\rightarrow\infty}\frac1t\log W_2(\pi_t^\mu,\pi_t^\nu)<0 \text{ for almost every sample path of }(Z_t^\dag)_{t\geq 0}.\]

Finally, combining Theorem~\ref{thm:gauss_lyapunov;cts} with asymptotic stability of the optimal filtering distribution $(\bar\pi_t^\nu)_{t\geq 0}$ due to V.H. Nguyen and the second named author~\cite{NguyenPark2026}, we show in Corollary~\ref{cor:MFEnKF_accurate;cts} that the MFEnKF is asymptotically accurate in the 2-Wasserstein distance: for suitable $\mu,\nu\in\P(\R^d)$ satisfying mild conditions,
\[\limsup_{t\rightarrow\infty}\frac1t\log W_2(\pi_t^\mu,\bar\pi_t^\nu)<0 \text{ for almost every sample path of }(Z_t^\dag)_{t\geq 0}.\]
Moreover, the exponent for the asymptotic exponential rate is characterized in terms of the limiting (continuous-time) Kalman gain.

\subsection{Literature}\label{ssec:literature}

Stratonovich~\cite{Stratonovich59,Stratonovich60conditional} initiated the study of the recursive procedure for the evolution of the conditional law~\eqref{def:barpi_n}, and Kushner~\cite{Kushner64} derived the stochastic partial differential equation characterizing the conditional law in continuous time, now known as the Kushner-Stratonovich equation. Duncan~\cite{Duncan68evaluation,Duncan70absolute}, Mortensen~\cite{Mortensen1966optimal}, and Zakai~\cite{Zakai69optimal} independently derived the linear SPDE satisfied by the unnormalized conditional law, the Duncan-Mortensen-Zakai equation; for simplicity, we refer to this equation as the Zakai equation throughout the paper. The well-posedness theory of these equations was then developed notably by Pardoux~\cite{Pardouxt80stochastic,Pardoux1982equations} and Krylov and Rozovskii~\cite{KrylovRozovskii77,KrylovRozovskii78,KrylovRozovskii82}. The pathwise or robust theory of filtering was initiated by Clark~\cite{Clark1978pathwise}. We refer the interested readers to the book of Bain and Crisan~\cite[Chapter 1.3]{BainCrisan09} for a detailed mathematical and historical account.

Asymptotic stability of the filter with respect to the initial data can be thought of as the well-posedness of the filtering problem, and is of fundamental importance: it implies that, regardless of the initialization, the filter will eventually reach the same conclusion as if it was initialized with the true initial law. Van Handel~\cite{vanHandel09stability} showed that, in broad generality, asymptotic stability of the signal process with nondegenerate observation noise implies that the filter is stable; see also the earlier works of Budhiraja~\cite{Budhiraja2003} Baxendale, Chigansky, and Lipster~\cite{BaxChiPav04Wonham}. On the other hand, one intuitively expects filter stability when the system is observable -- i.e. different initial laws should produce distinct laws of the observation process. Van Handel showed that filter stability follows from observability on compact state space~\cite{vanH09PTRF} and from the stronger condition of uniform observability on noncompact state spaces~\cite{vanHandel09uniform}; the latter applies to our linear coefficient setting when $(F,H)$ satisfies the observability condition.

Works on the quantitative stability of the optimal filter are more limited. Exponential asymptotic stability was established by Atar and Zeitouni on finite~\cite{AtarZeitouni97lyapunov} and compact~\cite{AtarZeitouni97exp} state spaces, where compactness leads to ergodicity of the signal, and by Atar~\cite{Atar98AoP} for noncompact state space under full observations. For nonergodic signals and partial observations, some quantitative stability results for nonergodic signals are available with strong assumptions on the initial data; we mention the papers of Stannat~\cite{Stannat05time-dep,Stannat06stability} based on a variational approach and the ground-state transform, and the thesis of van Handel~\cite[Chapter 4.1]{vanHandel06Thesis} using a probabilistic approach.

From connections to (linear) systems theory, one might expect filter stability to hold under a weaker condition than observability, namely detectability; this, roughly speaking, only asks for observability in the unstable directions of the signal process. The short and illuminating survey~\cite{vanHandel10proceedings} elaborates on connections between filter stability and systems theory. Asymptotic stability of the optimal filter under detectability is known only in the case of finite state space~\cite{vanH09PTRF} or the linear-Gaussian setting (namely the Kalman filter). We also mention the work of Mitter and Newton~\cite{MitterNewton03_variational} on the connection between filtering and optimal control via the variational characterization of the Bayes formula.

The Kalman filter, introduced in the pioneering work of Rudolf Kalman~\cite{Kalman60}, is a special case of the optimal filter in the linear-Gaussian setting; its continuous-time counterpart, the Kalman-Bucy filter, was introduced by Kalman and Bucy~\cite{KalmanBucy61}. 
Bucy established the global well-posedness and asymptotic stability of the covariance matrices under the assumptions of controllability and observability~\cite{Bucy67global}, and Wonham analyzed the equation under the weaker conditions of stabilizability and detectability~\cite{Wonham68riccati}. There are also several
results of a more quantitative nature: Anderson~\cite{anderson71stability} used the time-varying Lyapunov function to obtain exponential convergence rate of the covariance matrices, and Bougerol~\cite{Bougerol93} showed that the Riccati recursions are strict contractions in a natural Riemannian metric for positive definite matrices. 

Ocone and Pardoux~\cite{OconePardeoux96asymptotic} proved the asymptotic stability of the Kalman-Bucy filter in the probabilistic sense, namely that two Kalman-Bucy filter distributions with different initial data converge weakly (as probability measures) almost surely. Moreover, they derived the asymptotic exponential rates, where the exponent is (essentially) given by the spectral abscissa of the limiting observer system. Moreover, they adapted the techniques of Makowski~\cite{Makowski86} (see also Makowski and Sowers~\cite{SowersMakowski90}) to the continuous-time setting to treat non-Gaussian initial data. Recently, these results were extended by V.H. Nguyen and the second named author; this work establishes an asymptotic exponential decay (with the same exponent) of the total variation and $2$-Wasserstein distance between the optimal filtering distributions under the same conditions. When one of the optimal filtering distributions
is initialized at a specific initial law, the analogous result in the $\chi^p$-divergence is also available, for any finite $p\geq 1$.

The ensemble Kalman filter was introduced by Evensen in his seminal work~\cite{Evensen94EnKF}, and developed further in~\cite{vanLeeuwenEvensen96data,BurLeeuwenEvensen98analysis,Houtekamer98data}. 
Since its introduction, EnKF has seen great empirical success with a wide range of applications in geosciences, including oceanography, oil reservoir simulation, and weather forecasting. However, its theoretical understanding remains limited. For the history of related ensemble-based methods as well as their empirical success in various state estimation and Bayesian inverse problems, we refer to the literature survey in Calvello, Reich and Stuart's review~\cite{CalReiStu25}. Early theoretical work focused on showing convergence of the finite particle EnKF, as implemented in practice, to
the MFEnKF limit \cite{LeGland11large,mandel2011convergence}, in the linear-Gaussian setting where the MFEnKF distribution coincides with the Kalman filtering distribution.
Recent work has connected EnKF and the MFEnKF to the optimal filter \cite{CarHofStuVaes24,calvello2026accuracy} 
for problems close to linear and Gaussian.
However, these papers all consider finite time horizon results. Outside of the linear-Gaussian setting, works on the large-time behavior of the EnKF or MFEnKF are scarce. For fully observed systems --i.e. $H=I_d$ -- Kelly, Law, and Stuart~\cite{KellyLawStu14} established the well-posedness and accuracy (in the sense of a state estimator) of the MFEnKF, while De Wiljes, Reich, and Stannat studied the long-time stability and accuracy (also in the sense of a state estimator)  of the EnKF in continuous time~\cite{DeWilReiStannat18long-time}. Majda, Tong, and Kelly developed a stability theory for EnKF and related filters in terms of the observability energy criterion~\cite{MajdaTongKelly16,MajdaTongKelly15cov_inflation}, which appears to be difficult to verify outside of fully observed systems. The quantitative stability results for linear signal and observation processes in the interesting work of Del Moral and Tugaut~\cite{DelMoralTugaut18stability} seem most similar in spirit to our work, but they only consider the continuous-time setting, and the nature of their results differs from ours in that we focus on pathwise strictly contractive bounds with non-Gaussian initial data; see Remark~\ref{rmk:DelMoral} for a more detailed discussion. For a comprehensive overview on the theoretical aspects of EnKF, we refer the reader to the recent survey by Bishop and Del Moral~\cite{BishopDelMoral23survey}.

Our work is inspired in part by the work of Carrillo and Vaes~\cite{CarVaes21EKI} on the continuous-time mean-field ensemble Kalman sampler, an algorithm arising in the context of Bayesian inverse problems~\cite{G-IHWS20EKI}. In the setting where the Bayesian posterior distribution is Gaussian but the initial data need not be, Carrillo and Vaes establish an exponential convergence rate in the classical 2-Wasserstein distance; in their analysis, they utilize the Wasserstein distance between linearly transformed measures, which is more general than the matrix-weighted Wasserstein distances used in this paper. Burger, Erbar, Hoffmann, Matthes, and Schlichting~\cite{BEHMS25CovOT} provided a detailed analysis of the covariance-modulated optimal transport distance, with respect to which the ensemble Kalman sampler is a gradient flow, and refined the convergence analysis. In fact, the matrix-weighted Wasserstein distance we use in this paper can be thought of as a simpler relative of the covariance-modulated optimal transport distance, where the simplicity allows for an easier study of the geodesic and related convexity properties. The relationship between the ensemble Kalman sampler and EnKF stems from a similar way of approximating the gradient of the observation map in the inverse problem step, resulting in a covariance pre-conditioned gradient in the linear case. However, unlike the ensemble Kalman sampler, the ensemble Kalman filter does not have a natural gradient flow structure, and in general exhibits more complicated dynamics; the EnKF accounts for both the signal dynamics and the Bayes update, and need not stabilize to an equilibrium even when the filter is asymptotically stable in the initial data. The potential to employ covariance-weighted Wasserstein distances in the study of ensemble Kalman filters was first noted by Reich and Cotter~\cite{reich2015probabilistic}. 

Preliminary results on the matrix-weighted Wasserstein distances established in Section~\ref{ssec:W2M} rely on standard arguments in the classical theory of optimal transport and gradient flows; we do not attempt to capture the history of these fields here, but we mention a few most relevant works.
In their seminal work, Jordan, Kinderlehrer and Otto \cite{JKO98} observed that the Fokker-Planck equation can be obtained as a limit of an iterative variational scheme in the Wasserstein space. This scheme is an instance of the minimizing movement scheme due to De Giorgi~\cite{DeGiorgi93}, inspired by the work of Almgren, Taylor, and Wang~\cite{ATW93}. McCann identified the geodesics in the Wasserstein space via displacement interpolation and introduced the notion of displacement convexity of energy functionals~\cite{McCann97}. Otto 
and Villani~\cite{OttoVillani00} related the geodesic convexity and various functional inequalities. Ambrosio, Gigli, and Savar\'{e}~\cite{AGS} synthesized a general and robust theory of gradient flows in metric spaces, with refined results on the Wasserstein gradient flows.

As our study of the matrix-weighted Wasserstein distances is motivated by the need to approximate the Kullback-Leibler divergence in anisotropic settings, we mention a few related works. In order to justify the use of covariance as the matrix weight, we use the observation made by Chewi and Pooladian~\cite{chewi2023entropic} on covariance matrices via Brascamp-Lieb~\cite{BrascampLieb76} inequality and Cramer-Rao bounds~\cite{Cramer46}. Bregman transport cost studied by Cordero-Erausquin~\cite{Cordero17} and later by Chewi and Ahn~\cite{AhnChewi21Bregman} provides a more general and accurate approximation of the KL-divergence than by covariance-weighted Wasserstein distances; see Remark~\ref{rmk:BregmanT}. Recent works of Shenfeld~\cite{Shenfeld24matrix} and Aishwarya, Rotem, and Shenfeld~\cite{AiRoShenfeld25} introduce and study the matrix displacement convexity, which is an anisotropic generalization of displacement convexity.

\section{A Variational Derivation of the Mean-Field EnKF}\label{sec:EnKF;var_der}

In this section, we study a variational approximation of the Bayesian inverse problem using the matrix-weighted Wasserstein distances, in particular providing a variational derivation of the mean-field ensemble Kalman filter. 
As this derivation only concerns the analysis step of MFEnKF, the signal dynamics play no role in this section. Consequently, the results of this section apply regardless of the linearity of the signal process. 

In Section~\ref{ssec:W2M} we introduce matrix-weighted 2-Wasserstein distances and state Proposition~\ref{prop:funct_ineq;W2M}, which comments on the relationship of the distances with the Kullback-Leibler divergence. 
In Lemma~\ref{lem:EL;potential} we study the Euler-Lagrange equation of the steepest descent of the potential energy with respect to the matrix-weighted Wasserstein distance. Proofs of Proposition~\ref{prop:funct_ineq;W2M} and Lemma~\ref{lem:EL;potential} 
are placed in Appendix~\ref{app:W2M}, as they are technical and not needed in the remainder of this manuscript.

Building on these results, Section~\ref{ssec:bayes;approx} first motivates the use of covariance-weighted Wasserstein distance in the variational problem approximating the Bayes update. Then we establish the main result of this section, Theorem~\ref{thm:MFEnKF;var}, which states that in the case of linear observation, the said variational problem has a unique solution, and that the solution coincides with the analysis step of the MFEnKF. From Lemma~\ref{lem:EL;potential} readers might already expect that this is the case, but the uniqueness of the minimizer will require some care, making use of the strict convexity of the variational problem.
In Section~\ref{ssec:disc-cts}, we describe the recursive process of the MFEnKF including the forecast step. We focus on exploring the connection between the discrete and continuous-time MFEnKF, where the former is a time discretization of the latter. Clarifying the connection between the discrete and continuous-time dynamics will not only help compare the results in the two settings in Section~\ref{sec:linear}, but also make the gradient flow structure of the analysis step more apparent.

\subsection{Matrix-weighted 2-Wasserstein Distances}\label{ssec:W2M}
Given a nonnegative cost function $c:\R^d\times\R^d\rightarrow\R_+$, define the optimal transport cost $\OT_c:\P(\R^d)\times\P(\R^d)\rightarrow[0,+\infty]$ by
\begin{equation}\label{def:OT}
        \OT_c(\mu,\nu) :=  \inf_{\gamma\in\Pi(\mu,\nu)}\iint_{\R^d\times \R^d} c(x,y) \,d\gamma(x,y)
\end{equation}
where $\Pi(\mu,\nu)$ is set of couplings of $\mu,\nu$, namely if $\id^1,\id^2:\R^d\times\R^d\rightarrow\R^d$ are the projection maps to the first and second $d$-dimensional coordinates,
\begin{equation}\label{def:Gamma}
    \Pi(\mu,\nu)=\left\{\gamma\in\P(\R^d\times\R^d):\;\id^1_\#\gamma=\mu \text{ and } \id^2_\#\gamma=\nu\right\}.
\end{equation}
When the cost function is given by the square of the Euclidean norm $c(x,y)=|x-y|^2$, $\OT_c(\mu,\nu)=W_2^2(\mu,\nu)$ where $W_2$ is the well-known 2-Wasserstein distance.

The matrix-weighted 2-Wasserstein distance we use is an anisotropic generalization of $W_2$. Fix a symmetric positive definite matrix $M\in\R^{d\times d}_{++}$. Then define the matrix-weighted 2-Wasserstein distance $W_{2}(\cdot,\cdot;M):\P_2(\R^d)\times\P_2(\R^d)\rightarrow \R_+$, which we will often abbreviate as $W_{2,M}$, by
\begin{equation}\label{def:W2M}
    W_2(\mu,\nu;M)=\inf_{\gamma\in\Pi(\mu,\nu)}\left(\iint_{\R^d\times\R^d} |x-y|_M^2\;d\gamma(x,y)\right)^{1/2},
\end{equation}
where $\Pi(\mu,\nu)$ denotes the set of all couplings of $\mu,\nu$. 
We denote by $\Pi_{o}(\mu,\nu;M)$ the set of optimal couplings of $\mu,\nu$ for the problem \eqref{def:W2M}
\begin{equation}\label{def:Pio;M}
    \Pi_{o}(\mu,\nu;M)=\left\{\gamma\in\Pi(\mu,\nu):\;\iint_{\R^d\times\R^d}|x-y|_M^2\,d\gamma(x,y)=W_2^2(\mu,\nu;M)\right\}.
\end{equation}
From the standard theory of optimal transport it follows that $\Pi_o(\mu,\nu;M)$ is nonempty for any $M\in\R^{d\times d}_{++}$ and $\mu,\nu\in\P_2(\R^d)$. When there exists $T_{\mu}^{\nu}:\R^d\rightarrow\R^d$ such that $(\id\times T_{\mu}^{\nu})_\#\mu\in\Pi_o(\mu,\nu;M)$, where $\id:\R^d\rightarrow\R^d$ is the identity map on $\R^d$, we say $T_{\mu}^{\nu}$ is an $M$-weighted optimal transport map from $\mu\in\P_2(\R^d)$ to $\nu\in\P_2(\R^d)$.  Moreover, we can use the optimal transport maps (or more generally optimal transport plans) to construct constant-speed geodesic $(\mu_t)_{t\in[0,1]}$ such that $\mu_0=\mu$, $\mu_1=\nu$ and
\[W_2(\mu,\mu_t;M)=t W_2(\mu,\nu;M).\]

The matrix-weighted Wasserstein distance \eqref{def:W2M} is motivated by the need to better capture the convexity of energy functionals that may be anisotropic. An energy functional $\Ecal:\P_2(\R^d)\rightarrow(-\infty,+\infty]$ is said to be $\lambda$-geodesically convex along $W_2$-geodesics if the constant-speed $W_2$-geodesic $(\mu_t)_{t\in[0,1]}$ between any $\mu,\nu\in\P_2(\R^d)$ satisfies
\begin{equation}\label{def:W2_geocvx}
    \Ecal(\mu_t)\leq (1-t)\Ecal(\mu)+t\Ecal(\nu)-\frac{\lambda t(1-t)}{2}W_2^2(\mu,\nu).
\end{equation}
For instance, functionals $\Ecal$ of the form $\mu\mapsto\int_{\R^d} V(x)\;d\mu(x)$ with $D^2 V\succeq \lambda I_d$ satisfy \eqref{def:W2_geocvx}. However, as $\lambda$ is a lower bound for the smallest eigenvalue of the Hessian of $V$, it does not adequately capture the convexity of $V$ when it is highly anisotropic. We would like to use the matrix-weighted Wasserstein distances to improve in this regard. Specifically we seek energy
functionals satisfying
\begin{equation}\label{def:W2_geocvx2}
    \Ecal(\mu_t)\leq (1-t)\Ecal(\mu)+t\Ecal(\nu)-\frac{t(1-t)}{2}W_2^2(\mu,\nu;M).
\end{equation}

Now consider a probability measure $\pi(dx)\propto e^{-U(x)}\;dx$ with $D^2 U\succeq M^{-1}$. It is well known that we can write the Kullback-Leibler divergence with respect to $\pi$, $\kl(\cdot||\pi)$,
\[
    \kl(\mu||\pi):=\int_{\R^d}\log\frac{d\mu}{d\pi}\;d\mu=\int_{\R^d}U(x)\;d\mu(x)+ \Ent(\mu) +\const \text{ if }\mu\ll\pi.
\]
Define also the $M$-weighted Fisher Information
\begin{equation}\label{def:IM}
    I_M(\mu||\pi):=\int_{\R^d} \left|M\nabla \log\frac{d\mu}{d\pi}\right|_M^2\;d\mu \text{ if } \mu\ll\pi.
\end{equation}
Here, the logarithmic gradient is understood in the weak sense in $L^2(\mu;\R^d)$, and $I_M(\mu||\pi)$ is set to be $+\infty$ if $\mu\not\ll\pi$, or if the logarithmic gradient is not in $L^2(\mu;\R^d)$.
As the modulus of geodesic convexity is additive, Lemmas~\ref{lem:potential;cvx}-\ref{lem:entropy;cvx} imply that $\kl(\cdot||\pi)$ is $1$-convex along $W_{2,M}$-geodesics. Thus we can deduce the following anisotropic version of Talagrand's transport inequality~\cite{OttoVillani00,Talagrand96T2}.

\begin{proposition}\label{prop:funct_ineq;W2M}
Let $U\in \C^2(\R^d)$ satisfy $\nabla ^2 U\geq M^{-1}$ for some $M\in\R^{d\times d}_{++}$, and let $\pi(dx)\propto e^{-U(x)}dx$ be a probability measure with bounded second moments. Then $\kl(\cdot||\pi):\P_2(\R^d)\rightarrow(-\infty,+\infty]$ is $1$-convex along $W_{2,M}$-geodesics: for any constant speed $W_{2,M}$-geodesic $(\mu_t)_{t\in[0,1]}$ from $\mu_0$ to $\mu_1$, both in $\P_2(\R^d)$, 
\begin{equation}\label{eq:KL;cvx}
    \kl(\mu_t||\pi)\leq (1-t)\kl(\mu_0||\pi)+t\kl(\mu_1||\pi)-\frac{t(1-t)}{2}W_2^2(\mu_0,\mu_1;M).
\end{equation}
Moreover, the Kullback-Leibler divergence satisfies the analogue of Talagrand's transport inequality
\begin{equation}\label{eq:Talagrand;W2M}
    \frac12 W_2^2(\mu,\pi;M)\leq \kl(\mu||\pi) \text{ for any }\mu\in\P_2(\R^d),
\end{equation}
and the HWI-type inequality
\begin{equation}\label{eq:HWI;W2M}
    \kl(\mu||\pi)\leq W_{2}(\mu,\pi;M)\sqrt{I_M(\mu||\pi)}-\frac{1}{2}W_2^2(\mu,\pi;M) \text{ for any }\mu\in\P_2(\R^d),
\end{equation}
where we understand both sides to be $+\infty$ when $\mu\not\ll\pi$.
In particular, \eqref{eq:HWI;W2M} implies the log-Sobolev-type inequality
\begin{equation}\label{eq:LSI;W2M}
    \kl(\mu||\pi)\leq \frac{1}{2}I_M(\mu||\pi) \text{ for any }\mu\in\P_2(\R^d).
\end{equation}
\end{proposition}

\begin{proof}
    See Appendix~\ref{app:W2M}.
\end{proof}

\begin{remark}\label{rmk:funct_ineq}
    Combining \eqref{eq:Talagrand;W2M}-\eqref{eq:HWI;W2M}, we obtain
    \[\frac12 W_2^2(\mu,\pi;M)\leq \kl(\mu||\pi)\leq \frac12 W_2(\mu,\pi;M)\left(2\sqrt{I_M(\mu||\pi)}-W_2(\mu,\pi;M)\right).\]
    The inequalities suggest that the approximation of $\kl(\mu||\pi)$ by $\frac12W_{2,M}^2(\mu,\pi)$ becomes less accurate as $\sqrt{I_M(\mu||\pi)}$ deviates further from $W_{2,M}(\mu,\pi)$. This can happen in at least two different ways: by $\pi$ becoming less log-concave or $\mu$ becoming irregular. Firstly, note that as $M$ (hence its eigenvalues) grows larger, $I_M$ becomes larger but $W_{2,M}$ smaller. Thus, when $U$ becomes less convex in some directions, the corresponding eigenvalues of $M$ grow and the gap becomes bigger. Also, when $\pi$ is far away from Gaussian, so that the Hessian of its (negative) logarithm varies largely depending on $x$, the lower bound with a constant matrix $M^{-1}$ can only become poorer. Secondly, note that the Fisher information $I_M$ sees the first derivatives of the log density whereas $W_{2,M}$ behaves like the $\dot H^{-1}$-norm~(see \cite{RPeyre18}). Thus, the gap grows larger as $\mu$ becomes more irregular.
\end{remark}

We now study the following variational problem 
\begin{equation}\label{eq:JKO;V}
    \pi^\ast \in\argmin_{\mu\in\P_2(\R^d)} \Bigl(\frac12 W_{2,M}^2(\mu,\pi)+\int_{\R^d} V(x)\;d\mu(x)\Bigr).
\end{equation}
The resulting map $\pi \to \pi^\ast$ can be understood as defining
an implicit Euler scheme for the gradient flow of the potential energy functional $\mu\mapsto \int V\;d\mu$ with respect to a metric defined by
$W_{2,M}$. We will characterize (any) minimizer of \eqref{eq:JKO;V} by a suitable optimal transport map.

We are interested in the problem of the form \eqref{eq:JKO;V} as an approximation (given a suitable matrix $M\in\R_{++}^{d\times d}$) of the variational formulation of the Bayes update, setting $\pi$ to be the prior,
\[
\pi_{post}\in\argmin_{\mu\in\P_2(\R^d)}\Bigl(\kl(\mu||\pi)+\int_{\R^d}V(x;y)\;d\mu(x)\Bigr);\]
in this context, the potential $V(\cdot;y)$ is the negative-log-likelihood of a random variable given data $y$. In Section~\ref{ssec:bayes;approx}, we will explore this connection more thoroughly.

\begin{lemma}[{\bf Euler-Lagrange Equation for Potential Energy}]\label{lem:EL;potential}
    Let $V\in \C^1(\R^d)$ be bounded from below. Fix $M\in\R^{d\times d}_{++}$ and $\pi\in\P_2(\R^d)$, and let $\pi^\ast$ be a minimizer of the variational problem \eqref{eq:JKO;V}. Then the $W_{2,M}$-optimal transport map $T_{\pi^\ast}^{\pi}$ from $\pi^\ast$ to $\pi$ exists, and is given by
    \begin{equation}\label{eq:EL;potential}
        T_{\pi^\ast}^{\pi}(x)-x=M\nabla V(x).
    \end{equation}
\end{lemma}

\begin{proof}
   See Appendix~\ref{app:W2M}.
\end{proof}

\begin{remark}\label{rmk:BregmanT}
    Given a convex potential $U:\R^d\rightarrow\R$, Cordero-Erausquin~\cite{Cordero17} showed that the Bregman transport cost defined by
    \[\D_U(\mu,\nu):=\inf_{\gamma\in\Pi(\mu,\nu)}\iint_{\R^d\times\R^d}\left(U(x)-U(y)-\nabla U(y)\cdot(x-y)\right)\;d\gamma(x,y)\]
    satisfies, for $\pi(dx)\propto \exp(-U(x))\;dx,$
    \[\D_U(\mu,\pi)\leq \kl(\mu||\pi);\]
    see also Ahn and Chewi~\cite{AhnChewi21Bregman}. If $D^2U(x)\succeq M^{-1}$ then $D_U(\mu,\pi)\geq\frac12 W_{2,M}^2(\mu,\pi)$, and so the above inequality provides a tighter approximation than~\eqref{eq:Talagrand;W2M} by making use of the local curvature information of $U$. Moreover, the Bregman transport cost does not require one to choose $M$.

    As $\D_U$ is an optimal transport cost with asymmetric cost function $c(x,y)=U(x)-U(y)-\nabla U(y)\cdot(x-y)$ (which is nonnegative when $U$ is convex), rigorously establishing the Euler-Lagrange equation for
    \[\pi^\ast\in\argmin_{\mu\in\P_2(\R^d)}\Bigl(\D_U(\mu,\pi)+\int_{\R^d} V(x)\;d\mu(x)\Bigr)\]
    would require more technical work. However, when $\pi^\ast$ is close to $\pi$, which is the case when the variational problem comes from a time discretization of a continuous-time problem, then the optimal plan $\gamma$ defining $\D_U$ concentrates around the diagonal. Hence, on the support of $\gamma$,
    \[U(x)-U(y)-\nabla U(y)(x-y)\approx \frac12(x-y)^TD^2 U(y)(x-y).\]
    Thus we can formally expect
    \[T_{\pi^\ast}^\pi(x)-x\approx \left[D^2U(y)\right]^{-1}\nabla V(x) \text{ where } y=T_{\pi^\ast}^\pi(x).\]
\end{remark}

\subsection{A Variational Approximation of the Bayesian inverse problem}\label{ssec:bayes;approx}

Let us now turn to the problem of approximating the Bayes inverse problem by the use of covariance-weighted Wasserstein distances. Denote the prior by $\hat\pi$, the true Bayes posterior by $\bar\pi$, and its MFEnKF-type approximation by $\pi$, all elements of $\P_2(\R^d)$. Suppose the observation $Y$ is related to state $\hat X\sim \hat\pi$ by
\begin{equation}\label{def:Y=hx+eta}Y=H\hat X+\eta,\end{equation}
where $\eta\sim N(0,\Gamma)$ for some $\Gamma\in\R^{k\times k}_{++}$. In particular, the likelihood of $Y=y$ given $\hat X$ is proportional to $\exp\left(-\frac12|y-Hx|_{\Gamma}^2\right)$. Thus, given the observation $Y=y$, the Bayes update defines the posterior as
\begin{equation}\label{eq:posterior}
    \bar\pi(dx):=\B(\hat\pi;y)=\frac{\exp\left(-\frac12|y-Hx|_\Gamma^2\right)\hat\pi(dx)}{\int_{\R^d}\exp\left(-\frac12|y-Hx'|_\Gamma^2\right)\hat\pi(dx')}.
\end{equation}
As is well known in Variational Bayesian methods, $\bar\pi$ can be characterized as 
\begin{equation}\label{eq:posterior;variational}
\bar\pi=\argmin_{\mu\in\P_2(\R^d)} \Bigl(\kl(\mu|\hat\pi)+\frac12\int_{\R^d}|Hx-y|_\Gamma^2\;d\mu(x)\Bigr).
\end{equation}

We will show in this section that the analysis step of the MFEnKF corresponds to the following variational problem:
\begin{equation}\label{eq:W2Mpost;variational}
    \pi\in\E^{\eta}\left[\argmin_{\mu\in\P_2(\R^d)}\Bigl(\frac{1}{2}W_2^2(\mu,\hat\pi;\Cov(\hat\pi))+\frac12\int_{\R^d}|Hx+\eta-y|_\Gamma^2\;d\mu(x)\Bigr)\right] \text{ where }\eta\sim N(0,\Gamma).
\end{equation}
More precisely, we can define the minimizer for each realization of $\eta$
\begin{equation}\label{eq:W2Mpost;variational;eta}
    \pi_{\eta}\in\argmin_{\mu\in\P_2(\R^d)}\Bigl(\frac{1}{2}W_2^2(\mu,\hat\pi;\Cov(\hat\pi))+\frac12\int_{\R^d}|Hx+\eta-y|_\Gamma^2\;d\mu(x)\Bigr),
\end{equation}
and define $\pi$ as the expectation
\begin{equation}\label{eq:pi=avgpi^eta}
\pi(\varphi):=\E^{\eta}[\pi_{\eta}(\varphi)]=\E^{\eta,\hat X}[\varphi(T_{\hat\pi}^{\pi_{\eta}}(\hat X))] \text{ for each } \varphi\in\C_b(\R^d);
\end{equation}
here, $T_{\hat\pi}^{\pi_{\eta}}$ is the $\Cov(\hat\pi)$-weighted optimal transport map from $\hat\pi$ to $\pi_{\eta}$, and the expectation is with respect to the independent random vectors $\eta\sim N(0,\Gamma)$ and $\hat X\sim\hat\pi$. Of course, the existence and uniqueness of the minimizer $\pi_\eta$ in \eqref{eq:W2Mpost;variational;eta} must be verified to ensure that $\pi$ is well defined; this is the main content of Theorem~\ref{thm:MFEnKF;var}. 
Comparing \eqref{eq:W2Mpost;variational} with \eqref{eq:posterior;variational}, we notice two differences. Firstly, $\kl(\mu||\hat\pi)$ has been replaced with $\frac12W_2^2(\mu,\hat\pi;\Cov(\hat\pi))$. Secondly, observation noise $\eta$ has been added in the negative log-likelihood function.
We now discuss each of these modifications of the standard
variational formulation of the Bayes update from prior to posterior.

The reason for the first modification, in particular the choice of $M=\Cov(\hat\pi)$ is as follows. We have established in 
Proposition~\ref{prop:funct_ineq;W2M} (see Remark \ref{rmk:funct_ineq})
that, if $\hat\pi(dx)\propto\exp(-\hat U(x))\;dx$ and $D^2\hat U(x)\succeq M^{-1}\succ 0$, then
\[\frac12 W_2^2(\mu,\hat\pi;M)\leq \kl(\mu||\hat\pi)\leq \frac12 W_2(\mu,\hat\pi;M)\left(2\sqrt{I_M(\mu||\hat\pi)}-W_2(\mu,\hat\pi;M)\right).\]
Thus we would like $M$ to be a good approximation of $[D^2\hat U(x)]^{-1}$.  
When $\hat U\in\C^2(\R^d)$ is strictly convex, the Cramer-Rao~\cite{Cramer46,Rao45} and the Brascamp-Lieb inequalities~\cite{BrascampLieb76} state (see \cite[Lemma 7]{chewi2023entropic} for the statement in the form below)
\begin{equation}\label{eq:BL+CR;hatC}
    \left[\int D^2 \hat U(x) \;d\hat\pi(x)\right]^{-1} \leq \Cov(\hat\pi) \leq  \int \left[D^2 \hat U(x)\right]^{-1}\;d\hat\pi(x);
\end{equation}
the Cramer-Rao bound on the left does not require convexity of $\hat U$. This suggests that $M:=\Cov(\hat\pi)$ is a reasonable choice. This choice also has an additional computational advantage, as covariance matrices can be effectively approximated based on the ensemble of particles via the sample covariance. 
Furthermore, \eqref{eq:BL+CR;hatC} is an equality when $\hat\pi$ is a Gaussian, thus we can expect the approximation of $\kl(\cdot|\hat\pi)$ by $\frac12 W_2^2(\cdot,\hat\pi;\Cov(\hat\pi))$ to be more accurate, hence also the approximation of the variational problem \eqref{eq:posterior;variational} with this substitution.

The second modification ensures consistency with the observation model, given that the transport metric, unlike the KL-divergence, contains particle-wise correspondence information. More precisely, the prior particle location $x_{prior}\in\supp\hat\pi$ is, by the covariance-weighted optimal transport map, associated to the (approximate) posterior particle location $x_{post}\in\supp\pi_\eta$ -- i.e. $x_{prior}=T_{\pi_\eta}^{\hat\pi}(x_{post})$. Thus, by the Euler-Lagrange equation \eqref{eq:EL;potential}, we obtain the map
\[x_{prior}=x_{post}+\Cov(\hat\pi)H^T\Gamma^{-1}(Hx_{post}+\eta-y);\]
note that this map corresponds to the inverse of $T^{EK}$ in~\eqref{eq:hatB}, with $\mu=\hat\pi$ and $v=y-\eta$.
If $\eta=0$, given observation $y$ and the (approximate) `posterior particle location' $x_{post}$, the corresponding particle configuration $x_{prior}$ can be precisely recovered. This is unnatural; as we see in the definition of the observation variable \eqref{def:Y=hx+eta}, even with the observation $y$ we should only be able to recover $Hx_{prior}$ up to the corruption by additive noise $\eta\sim N(0,\Gamma)$; and of course recovery of $x_{prior}$ can only be harder as it requires inversion of (potentially rank-deficient) $H$. Note that this is due to the implicit and Lagrangian nature of the Euler-Lagrange equation associated to \eqref {eq:W2Mpost;variational;eta}; in particular, the Bayes formula does not retain such particle-wise information through a canonical coupling, thus the same argument does not apply to the true posterior update \eqref{eq:posterior;variational}.
\medskip

We now establish the main result of this section.

\begin{theorem}\label{thm:MFEnKF;var}
    Let $H\in\R^{k\times d}$ and $\hat\pi\in\P_2(\R^d)$ with $\Cov(\hat\pi)\succ 0$. For each realization $y^\dag$ of the observation and observational noise $\eta\sim N(0,\Gamma)$, the variational problem
    \begin{equation}\label{eq:W2Mpost;variational;lin}
    \argmin_{\nu\in\P_2(\R^d)}\Bigl(\frac{1}{2}W_2^2(\nu,\hat\pi;\Cov(\hat\pi))+\frac12\int_{\R^d}|Hx+\eta-y^\dag|_\Gamma^2\;d\nu(x)\Bigr)
    \end{equation}
    has a unique solution $\pi_{\eta}\in\P_2(\R^d)$. Moreover, the $\Cov(\hat\pi)$-weighted optimal transport map $T_{\hat\pi}^{\pi_{\eta}}$ from $\hat\pi$ to $\pi_{\eta}$ is given by
    \begin{equation}\label{eq:T-hatpi-pi^eta;lin}
        T_{\hat\pi}^{\pi_{\eta}}(\hat x)= \hat x-\Cov(\hat\pi)H^T(H\Cov(\hat\pi)H^T+\Gamma)^{-1}(H\hat x+\eta-y^\dag).
    \end{equation}
Setting  $y^\dag=Y_{n+1}^\dag$, $\eta=\eta_{n+1}$, $\hat x=\hat X_{n+1}$, $\hat\pi=\hat\pi_{n+1}^\mu$ gives the update rule \eqref{eq:hatB} for the analysis step of the mean-field ensemble Kalman filter~\eqref{def:MFEnKF;disc}. 
\end{theorem}

\begin{proof}
    The existence of a minimizer follows from the direct method of the calculus of variations. Indeed, both the square of the $\Cov(\hat\pi)$-weighted Wasserstein distance and the potential energy are lower-semicontinuous  with respect to the narrow convergence of measures (i.e. weak convergence of measures against continuous and bounded test functions). Clearly, any minimizing sequence should have bounded second moments, hence precompact in the narrow topology by Prokhorov's theorem.

    Uniqueness is a consequence of the strict convexity of the minimization problem along the following generalized geodesic: given $\mu,\nu\in\P_2(\R^d)$ and $\Cov(\hat\pi)$-weighted optimal transport maps $T_{\hat\pi}^{\mu},T_{\hat\pi}^{\nu}$ from $\hat\pi$ respectively to $\mu$ and $\nu$, consider the interpolating curve $\mu_t:=((1-t)T_{\hat\pi}^{\mu}+t T_{\hat\pi}^{\nu})_\# \hat\pi$. Then one can verify $\mu\mapsto \frac12 W_2^2(\mu,\hat\pi;\Cov(\hat\pi))$ is 1-convex along this curve (see \cite[Section 9.2]{AGS}), whereas 
    \[D^2(|Hx+\eta-y^\dag|_\Gamma^2)=2H^T \Gamma^{-1} H\succeq 0\]
    implies that the potential energy is convex along this curve. The use of optimal transport maps in the argument is for notational convenience, and one can instead use optimal transport plans (couplings) to define the analogous generalized geodesics when the optimal transport maps do not exist; see \cite[Chapter 9.2]{AGS}.

    It remains to verify \eqref{eq:T-hatpi-pi^eta;lin}. Writing $\hat C:=\Cov(\hat\pi)$, the Euler-Lagrange equation~\eqref{eq:EL;potential} in this case reads
    \begin{align*}
    T_{\pi_{\eta}}^{\hat\pi}(x_{\eta})&=x_{\eta}+\hat CH^T \Gamma^{-1} (H x_{\eta}+\eta-y^\dag)   \\
    &=(I+\hat CH^T\Gamma^{-1}H)x_{\eta} + \hat CH^T\Gamma^{-1}(\eta-y^\dag).
    \end{align*}
    Thus we can invert this map to obtain
    \begin{align*}
        T_{\hat\pi}^{\pi_{\eta}}(\hat x):=(I+\hat C H^T\Gamma^{-1} H)^{-1}(\hat x-\hat C H^T\Gamma^{-1}(\eta-y^\dag)).
    \end{align*}
    We use the Woodbury matrix identity
    \begin{equation}\label{eq:woodbury}
    (A+UCV)^{-1}=A^{-1}-A^{-1}U(VA^{-1}U+C^{-1})^{-1}VA^{-1}
    \end{equation}
    in two forms: firstly, with the choice $A=I$, $U=\hat CH^T$, $C=\Gamma^{-1}$, $V=H$, which yields
\begin{equation}\label{eq:woodbury1}
    (I+\hat CH^T\Gamma^{-1}H)^{-1}
    = I - \hat CH^T(H\hat CH^T+\Gamma)^{-1}H,
\end{equation}
    and secondly, $A=U=\Gamma$, $V=H\hat C H^T$ and $C=\Gamma^{-1}$, which leads to
\begin{equation}\label{eq:woodbury2}
    (H\hat C\sans{H}^T +\Gamma)^{-1} = \Gamma^{-1} -(H \hat C H^T+\Gamma)^{-1}H\hat C H^T\Gamma^{-1}.
\end{equation}
Thus
\begin{align*}
    T_{\hat\pi}^{\pi_{\eta}}(\hat x)&=\hat x-\hat CH^T(H\hat CH^T +\Gamma)^{-1}H\hat x +\hat C H^T\Gamma^{-1}(y^\dag-\eta) \\
    &\qquad-\hat C H^T(H\hat C H^T +\Gamma)^{-1} H \hat C H^T \Gamma^{-1}(y^\dag-\eta)\\
    &=\hat x-\hat CH^T(H\hat CH^T +\Gamma)^{-1}H\hat x+\hat C H^T (\Gamma^{-1}-(H \hat C H^T+\Gamma)^{-1} H\hat CH^T\Gamma^{-1})(y^\dag-\eta) \\
    &=\hat x-\hat CH^T(H\hat CH^T +\Gamma)^{-1}H\hat x +\hat C H^T(H\hat C H^T +\Gamma)^{-1}(y^\dag-\eta)   \\
    &=\hat x+\hat C H^T ( H\hat C H^T +\Gamma)^{-1}(y^\dag-\eta-H\hat x).
\end{align*}
\end{proof}

\subsection{The Mean-field EnKF Recursion as a Time-Discretization}\label{ssec:disc-cts}
In light of Theorem ~\ref{thm:MFEnKF;var}, we have the following 
variational interpretation of the mean-field ensemble Kalman filter:
\begin{equation}\label{eq:MfEnKF;var;full}
\begin{split}
    \hat\pi_{n+1}&=\mathsf{P}\pi_n,\\
    \pi_{n+1}&=\E^{\eta_{n+1}}\left[\argmin_{\nu\in\P_2(\R^d)}
\Bigl(\frac12 W_2^2(\nu,\hat\pi_{n+1};\Cov(\hat\pi_{n+1}))+\frac12\int_{\R^d}|Hx+\eta_{n+1}-Y_{n+1}^\dag|_\Gamma^2\;d\nu(x)\Bigr)\right],
\end{split}
\end{equation}
where the expectation in the second step is defined as in \eqref{eq:pi=avgpi^eta}, and the operator $\mathsf{P}$ in the forecast step is given by \eqref{def:P}
\[\mathsf{P}\mu(dx)=\frac{1}{\sqrt{(2\pi)^d\det\Sigma}}\int_{\R^d}\exp\left(-\frac{|x-Ax'|^2_{\Sigma}}{2}\right)\mu(dx')\;dx;\]
note that $\mathsf{P}:\P_2(\R^d)\rightarrow\P_2(\R^d)$, hence the iterative process \eqref{eq:MfEnKF;var;full} remains in $\P_2(\R^d)$ once it is initialized there. In fact, the derivation in Section~\ref{ssec:bayes;approx} was independent of this specific choice of $\mathsf{P}$; thus, this 
characterization of the MFEnKF applies whenever the observation is linear 
and the observational noise is additive Gaussian, 
regardless of whether the signal process is linear. 

In many applications, the discrete-time signal and observation dynamics approximate continuous-time dynamics.
Consider the discrete-time model~\eqref{def:sig-obs;discrete}, repeated here for convenience:
\begin{align*}
    X_{n+1} &= AX_n + \xi_{n+1} \text{ where } \xi_{n+1}\sim N(0,\Sigma),\\
    Y_{n+1} &= HX_{n+1}+\eta_{n+1} \text{ where } \eta_{n+1}\sim N(0,\Gamma).
\end{align*}
We view this as a discretization of the continuous-time model~\eqref{eq:linearSigObs;cts}
\begin{align*}
    dX_t &= \sF X_t\;dt+\sqrt{\sSig}dB_t,\\
    dZ_t &= \sH X_t\;dt+\sqrt{\sG}dW_t,
\end{align*}
with the scaling \eqref{def:tau_scale}. Thus
$\sF\in\R^{d\times d},\sH\in\R^{k\times d},\sSig\in\R^{d\times d}_{++},\sG\in\R^{k\times k}_{++}$ and $(B_t)_{t\geq 0},(W_t)_{t\geq 0}$ are independent copies of the standard Brownian motion. Introducing the time step $\tau>0$ and setting \eqref{def:tau_scale} and $Y_n^\dag=Z_n^\dag-Z_{n-1}^\dag$, we recover the continuous-time model from the discrete-time model in the limit $\tau\to 0$; see \cite[Section 3.1]{CalReiStu25} for a more detailed discussion.

Let us fix $\tau>0$ and relabel the time indices $n\in\N$ by $n\tau$. Then $\sP=\sP_\tau$ becomes
\[\sP_\tau\mu(dx) = \frac{1}{\tau^{d/2}\sqrt{(2\pi)^d\det\sSig}}\int_{\R^d}\exp\left(-\frac{|x-(I+\tau\sF)x'|_{\sSig}^2}{2\tau}\right)\;\mu(dx')\;dx,\]
and simple calculations reveal that the optimal filter can be written as
\begin{equation}\label{eq:OF;var;tau}
\begin{split}
    \tilde\pi_{(n+1)\tau}&=\sP_\tau\bar\pi_{n\tau},\\
    \bar\pi_{(n+1)\tau}&=\argmin_{\mu\in\P(\R^d)}\Bigl(\frac{1}{\tau}\kl(\mu||\tilde\pi_{(n+1)\tau})+\frac12\int_{\R^d}\left|\sH x-\tfrac{Z_{(n+1)\tau}^\dag-Z_{n\tau}^\dag}{\tau}\right|_{\sG}^2\;d\mu(x)\Bigr);
\end{split}
\end{equation}
here, we have used $\tilde\pi$ for the intermediate measure to avoid conflict with the notation for the MFEnKF. Similarly, letting $\eta_{n+1}=\tau^{1/2}\sG^{1/2}g_{n+1}$ for $g_{n+1}\in N(0,I)$, and writing $\Cov(\hat\pi_{(n+1)\tau})=\hat C_{(n+1)\tau}$ for simplicity, the MFEnKF~\eqref{eq:MfEnKF;var;full} becomes
\begin{equation}\label{eq:MFEnKF;var;tau}
\begin{split}
    \hat\pi_{(n+1)\tau}&=\sP_\tau\pi_{n\tau},\\
    \pi_{(n+1)\tau}&=\E^{g_{n+1}}\argmin_{\mu\in\P(\R^d)}\Bigl[\frac{1}{2\tau}W_{2}^2(\mu,\hat\pi_{(n+1)\tau};\hat C_{(n+1)\tau})+\frac12\int_{\R^d}\left|\sH x+\tfrac{\sG^{1/2}g_{n+1}}{\tau^{1/2}}-\tfrac{Z_{(n+1)\tau}^\dag-Z_{n\tau}^\dag}{\tau}\right|_{\sG}^2\;d\mu(x)\Bigr].
\end{split}
\end{equation}
Having introduced the time step, the gradient flow structure (in the sense of the minimizing movement scheme) of the analysis steps in both \eqref{eq:OF;var;tau} and \eqref{eq:MFEnKF;var;tau} becomes more apparent.

Moreover, by making appropriate substitutions in~\eqref{eq:T-hatpi-pi^eta;lin}, we know that for each realization of $g_{n+1}$, the $\hat C_{(n+1)\tau}$-weighted optimal transport map $T_{(n+1)\tau}$ from $\hat\pi_{(n+1)\tau}$ to $\pi_{(n+1)\tau,g_{n+1}}$ is given by
\begin{align*}
    T_{(n+1)\tau}(\hat x)-\hat x&=-\hat C_{(n+1)\tau}(\tau\sH)^T(\tau^2\sH\hat C_{(n+1)\tau}\sH^T+\tau\sG)^{-1}\left(\tau\sH\hat x + \tau^{1/2}\sG^{1/2}g_{n+1} - (Z_{(n+1)\tau}^\dag-Z_{n\tau}^\dag)\right)\\
    &=-\tau\hat C_{(n+1)\tau}\sH^T(\tau\sH\hat C_{(n+1)\tau}\sH^T+\sG)^{-1}\left(\sH\hat x+\tau^{-1/2}\sG^{1/2}g_{n+1}-\tfrac{Z_{(n+1)\tau}^\dag-Z_{n\tau}^\dag}{\tau}\right).
\end{align*}
Let us write $\xi_{(n+1)\tau}=\tau^{1/2}\sSig^{1/2}g_{n+1}'$ for $g_{n+1}'\sim N(0,I)$ independent from $g_{n+1}$. 
Given $X_{n\tau}\sim\pi_{n\tau}$, the forecast step is given by
\begin{align*}
    \hat X_{(n+1)\tau}-X_{n\tau}=\tau\sF X_{n\tau}+\tau^{1/2}\sSig^{1/2} g_{n+1}'.
\end{align*}
Thus, the full MFEnKF update corresponds to
\begin{align*}
    X_{(n+1)\tau}-X_{n\tau}=&T_{(n+1)\tau}(\hat X_{(n+1)\tau})-\hat X_{(n+1)\tau}+\hat X_{(n+1)\tau}-X_{n\tau}\\
    =&-\tau\hat C_{(n+1)\tau}\sH^T(\tau\sH\hat C_{(n+1)\tau}\sH^T+\sG)^{-1}\left(\sH\hat X_{(n+1)\tau}+\tau^{-1/2}\sG^{1/2}g_{n+1}-\tfrac{Z_{(n+1)\tau}^\dag-Z_{n\tau}^\dag}{\tau}\right)\\
    &+\tau \sF X_{n\tau}+\tau^{1/2}\sSig^{1/2} g_{n+1}'
\end{align*}
As $H\hat X_{(n+1)\tau}=H X_{n\tau}+o(1)$, dividing both sides by $\tau$ and formally letting $\tau\to 0$ we obtain the continuous-time MFEnKF~\eqref{eq:EnKF;mf;cts} (cf. \cite[Section 3.1]{CalReiStu25})
\begin{equation*}
    d X_t=\sF X_t\;dt + \sqrt{\sans{\Sigma}}\;dB_t - \Cov(\pi_t)\sH^T\sG^{-1}(\sH X_t\;dt+\sqrt{\sG}dW_t- dZ_t^\dag).
\end{equation*}

Characterizations \eqref{eq:MfEnKF;var;full} and \eqref{eq:MFEnKF;var;tau} also hint at the mechanism of the large-time stabilization of the MFEnKF. As the variational problem in the second step is strictly convex in the observed subspace, we can expect contraction in $\range(\Cov(\hat\pi_{n+1})H^T)$, which can compensate for potential expansion of the forecast step.

The difficulty of analysis lies in the fact that this subspace and the modulus of convexity (hence the strength of contraction) depends on the covariance matrix of $\hat\pi_{n+1}$, which evolves with time. When the signal process is nonlinear, it is unclear whether understanding the evolution of the covariance matrix is any simpler than understanding the evolution of the measure \An{itself.
However,} for linear signal processes, the covariance matrices of the MFEnKF satisfy a closed system of matrix Riccati equations, \eqref{eq:Riccd;recursion} in discrete time and \eqref{eq:Riccc;diff} in continuous time, regardless of whether the initial data are Gaussian. Moreover, \eqref{eq:MfEnKF;var;full} and \eqref{eq:MFEnKF;var;tau} suggest that the covariance-weighted Wasserstein metric provides a natural way to measure the distance between the MFEnKF with the same covariance matrix. Leveraging these two facts, we study the large-time behavior of the MFEnKF in Section~\ref{sec:linear}.

\section{Large-time behavior in the linear setting with Gaussian noise}\label{sec:linear}
In this section we analyze the large-time behavior of the MFEnKF for linear signal and observation processes with additive Gaussian noise -- what
we term the linear setting with Gaussian noise -- both in discrete and continuous time. It is a remarkable property of the MFEnKF that the evolution of the covariance matrices is characterized by a closed system of matrix Riccati equations, for any initial measure.  Thus, by exploiting the covariance
weighted transport formulations of the MFEnKF developed in the preceding 
section, we can use the well-understood properties of the Riccati equations to understand the dynamics of the MFEnKF for general (non-degenerate) initial data.
The optimal filter, in contrast, only admits a closed evolution equation
for the covariance matrices when, additionally, 
the initial measure is Gaussian -- what we
term the linear-Gaussian setting. Our analysis focuses on the setting where the signal and observation matrices satisfy the detectability condition, which is the sharp condition ensuring asymptotic stability of the Riccati equations, and hence of the Kalman and Kalman-Bucy filters.

We first collect important properties of the algebraic Riccati equations in Section~\ref{ssec:prelim}. Propositions~\ref{prop:Riccati-d}-\ref{prop:Riccati-c} note that the Riccati equations govern the evolution of the covariance of the MFEnKF distributions in the linear setting with Gaussian noise, regardless
of initialization. As the solutions of the equations correspond to covariance matrices for the MFEnKF, we are interested in understanding when the solution converges to a unique nondegenerate steady state. 
The key conditions, namely detectability and controllability, are described in Definition~\ref{def:detectability}; then Proposition~\ref{prop:DARE} and Proposition~\ref{prop:CARE} state the asymptotic properties of solutions of the Riccati equations. 

Section~\ref{ssec:disc_time} then studies the behavior of the MFEnKF in discrete time. In Theorem~\ref{thm:gauss_lyapunov;disc} we show that the MFEnKF is a strict contraction at each step towards the Gaussian manifold $\P_2^g(\R^d)$ when measured in the covariance-weighted Wasserstein distances. We also include the corresponding asymptotic result Proposition~\ref{prop:gauss_conv;W2;asymp} in the usual $2$-Wasserstein distance to illustrate the difficulty of establishing nonasymptotic results under the same conditions.

In Section~\ref{ssec:EnKF;cts} we study the large-time behavior of the MFEnKF in continuous time. We first establish in Theorem~\ref{thm:gauss_lyapunov;cts} the nonasymptotic exponential contractive rates of the MFEnKF towards the Gaussian manifold, from which asymptotic stability of the MFEnKF (Theorem~\ref{thm:MFEnKF;stable;cts}) readily follows. Then we are able to use the asymptotic stability of the optimal filter due to V.H. Nguyen and Park~\cite[Corollary 4.9]{NguyenPark2026} to show that the 2-Wasserstein distance between the MFEnKF distribution and the optimal filtering distributions (with different initial laws) converge asymptotically exponentially almost surely; see Corollary~\ref{cor:MFEnKF_accurate;cts}.

\subsection{Preliminaries}\label{ssec:prelim}
In this section, we recall concepts from control theory that ensure that the solutions of the algebraic Riccati equations are well-behaved. Recall the discrete Riccati operator $\Ricc_{d}:\R^{d\times d}_+\rightarrow\R^{d\times d}_+$ defined in~\eqref{def:Riccd}:
\begin{equation*}
\Ricc_{d}(C)=A C A^T+\Sigma-A C H^T(H C H^T+\Gamma)^{-1}H C A^T,
\end{equation*}
where $A,\Sigma\in\R^{d\times d}, H\in\R^{k\times d}, \Gamma\in\R^{k\times k}$ with $\Sigma,\Gamma\succ0$. The analogous continuous Riccati operator $\Ricc_{c}:\R^{d\times d}_+\rightarrow\R^{d\times d}_{sym}$ is defined in \eqref{def:Riccc} as
\begin{equation*}
    \Ricc_{c}(C)=C\sF^T+\sF C+\sSig-C\sH^T\sG^{-1}\sH C,
\end{equation*}
where $\sF,\sSig\in\R^{d\times d}$, $\sH\in\R^{k\times d}$, $\sG\in\R^{k\times k}$ with $\sSig,\sG\succ 0$. For the connection between the discrete and continuous-time Riccati operators and the matrices involved, see Section~\ref{ssec:disc-cts}.

The Riccati equations govern the evolution of covariance matrices of the discrete-time and continuous-time MFEnKF distributions, regardless of the Gaussianity of the initial data. In particular, this implies that the covariance matrices are independent of observations, thanks to the linearity of the signal-observation model. Propositions~\ref{prop:Riccati-d}-\ref{prop:Riccati-c}
record these key observations about the covariance evolution for the
MFEnKF; both follow from direct and straightforward calculations, 
hence we omit their proofs.

\begin{proposition}\label{prop:Riccati-d}
Consider the discrete-time MFEnKF~\eqref{def:MFEnKF;disc} with deterministic observation sequence $\{y_n^\dag\}_{n\geq 1}$,
\begin{align*}
    \hat X_{n+1} &= AX_n+\xi_{n+1},\\
    X_{n+1} &=\hat X_{n+1}-\Cov(\hat \pi_{n+1}) H^T(H\Cov(\hat \pi_{n+1}) H^T+\Gamma)^{-1}(H\hat X_{n+1}+\eta_{n+1}-y_{n+1}^\dag),
\end{align*}
with i.i.d. sequences of signal and observation noise $\{\xi_n\}_{n\geq 1},\{\eta_n\}_{n\geq 1}$, independent from each other and $X_0$, satisfying $\Cov(\xi_n)=\Sigma\succ 0$, $\Cov(\eta_n)=\Gamma\succ 0$ for all $n\geq 1$. Then the covariance matrix $\hat C_n=\Cov(\hat \pi_n)$ of the predictive MFEnKF distribution satisfies the equation
\[\hat C_{n+1}=\Ricc_{d}(\hat C_n).\]
In particular, $\hat C_n$ is independent of the observations, and is deterministic.
\end{proposition}

\begin{proposition}\label{prop:Riccati-c}
Consider the continuous-time MFEnKF~\eqref{eq:EnKF;mf;cts}
\begin{align*}
    d X_t=\sF X_t\;dt + \sqrt{\sans{\Sigma}}\;dB_t - \Cov(\pi_t)\sH^T\sG^{-1}(\sH X_t\;dt+\sqrt{\sG}dW_t- dZ_t^\dag).
\end{align*}
Then the covariance matrix $C_t=\Cov(\pi_t)$ of the continuous-time MFEnKF distribution satisfies the equation
\[\dot C_t=\Ricc_{c}(C_t).\]
In particular, $\{C_t\}_{t\geq 0}$ is independent of the observations, and is deterministic.
\end{proposition}

To characterize stability properties of matrix dynamical systems we will use  the spectral radius $\rho(\cdot)$ from~\eqref{def:rho} in discrete time and the spectral abscissa $\alpha(\cdot)$ from~\eqref{def:alpha} in continuous time. We begin by defining the key conditions for stability of the Riccati recursion, namely detectability and controllability; the definitions are collected from Lancaster and Rodman~\cite[Chapter 4]{LanRod95}.
Recall the definitions \eqref{def:rho} and \eqref{def:alpha} of the spectral radius $\rho(\cdot)$ and spectral
abscissa $\alpha(\cdot)$.

\begin{definition}\label{def:detectability}
Let $A,M\in\R^{d\times d},B\in\R^{d\times k},H\in\R^{k\times d}$.
\begin{listi}
    \item (Stable matrix) A square matrix $M$ is $c$-stable if $\alpha(M)<0$; we say $M$ is $d$-stable if $\rho(M)<1$.

    \item (Stabilizable pair) Let $A\in\R^{d\times d}$ and $B\in\R^{d\times k}$. The pair $(A,B)$ is $c$-stabilizable (resp. $d$-stabilizable) if there exists $L\in\R^{k\times d}$ such that $A+BL$ is $c$-stable (resp. $d$-stable).

    \item (Detectable pair) Let $A\in\R^{d\times d}$ and $H\in\R^{k\times d}$. The pair $(H,A)$ is $c$-detectable (resp. $d$-detectable) if $(A^T,H^T)$ is $c$-stabilizable (resp. $d$-stabilizable), or, equivalently, if there exists a matrix $K\in\R^{d\times k}$ such that $A-KH$ is $c$-stable (resp. $d$-stable).

    \item (Controllable pair) Let $A\in\R^{d\times d}$ and $B\in\R^{d\times k}$ where $k\leq d$.
    We say the pair $(A,B)$ is controllable if
    \begin{equation}\label{def:controllable}
        \rank[B, AB, A^2B, \cdots , A^{d-1}B]=d.
    \end{equation}

    \item (Observable pair) Let $A\in\R^{d\times d}$ and $H\in\R^{k\times d}.$ We say the pair $(H,A)$ is observable if $(A^T,H^T)$ is controllable.
\end{listi}
\end{definition}
We will sometimes skip the prefixes $c$- and $d$- when they are clear from context.

\begin{remark}
It is well known that controllability of $(A,B)$ is a strictly stronger property than stabilizability of the same pair both in the continuous and discrete sense~\cite[Theorem 4.4.2]{LanRod95}. Indeed, under controllability, one can explicitly construct a suitable matrix $L$ to ensure the stability of $A+BL$. Equivalently, observability is strictly stronger than detectability.
\end{remark}

We now state our set of assumptions in the discrete-time and continuous-time settings.

\begin{assumption}\label{ass:disc}
We assume that $A,\Sigma\in\R^{d\times d}$, $H\in\R^{k\times d}$, and $\Gamma\in\R^{k\times k}$ in \eqref{def:sig-obs;discrete} satisfy the following assumptions:
  \begin{equation}\label{ass:detectable;d}
    \tag{A1d} (H,A) \text{ is a $d$-detectable pair,}
\end{equation}  
and
\begin{equation}\label{ass:nondegen;d}
    \tag{A2d} A\in\R^{d\times d} \text{ is invertible, and } \Sigma,\Gamma\succ0.
\end{equation}
\end{assumption}

In the continuous-time setting, we make the following assumptions.
\begin{assumption}\label{ass:cts}
We assume that $\sF,\sSig\in\R^{d\times d}$, $\sH\in\R^{k\times d}$, and $\sG\in\R^{k\times k}$ in \eqref{eq:linearSigObs;cts} satisfy the following assumptions:
\begin{equation}\label{ass:detectable;c}
    \tag{A1c} (\sH,\sF) \text{ is a c-detectable pair,}
\end{equation}
and
\begin{equation}\label{ass:nondegen;c}
    \tag{A2c} \sSig,\sG\succ 0.
\end{equation}
\end{assumption}

We make a few remarks on the assumptions.
\begin{remark}\label{rmk:assumptions}
Roughly speaking, we can understand detectability assumptions \eqref{ass:detectable;d} and \eqref{ass:detectable;c} as requiring the observations to be strong enough to compensate for potential instabilities in the signal dynamics. As we shall see, detectability is the key condition determining the existence of stable solutions to the Riccati equations in both discrete and continuous-time settings.

Regarding assumptions \eqref{ass:nondegen;d} and \eqref{ass:nondegen;c}, nondegeneracy of the observation noise covariance $\Gamma,\sG$ is necessary to rule out pathological behavior of the optimal filter. Without this condition, the optimal filter may be unstable even when the signal process is ergodic; see for instance~\cite[Section 3]{BaxChiPav04Wonham} for a simple counterexample in finite state space.

Nondegeneracy of the signal noise covariance $\Sigma,\sSig$ is related to controllability, which, together with detectability, guarantees the limiting covariance is nondegenerate. Hence we can understand controllability as a regularity condition. Positive definiteness of $\Sigma,\sSig$ implies controllability of $(A,\Sigma^{1/2})$ in the discrete case and $(\sF,\sSig^{1/2})$ in the continuous case without any additional condition on $A$ and $\sF$. Many results in this paper can be generalized to the case where the nondegeneracy assumptions on $\Sigma,\sSig$ are replaced with controllability of the said pairs; however, we impose the stronger assumptions to focus on the detectability condition which captures the relationship between signal and observation processes.

Finally, nondegeneracy of $A$ in the discrete-time case is a technical condition that, together with positive definiteness of $\Sigma$, allows us to obtain strict contraction of the MFEnKF towards the Gaussian manifold in Theorem~\ref{thm:gauss_lyapunov;disc}. When the filter comes from a time discretization of a continuous dynamics, $A$ takes the form $A=I+\tau\sF$, hence is nondegenerate when the time step is sufficiently small. This is consistent with 
the fact that the corresponding continuous-time result, Theorem~\ref{thm:gauss_lyapunov;cts}, requires $\sSig\succ0$ but not the invertibility of $\sF$.
\end{remark}

We now discuss the implication of these conditions for matrix Riccati equations. To this end, recall the notations \eqref{def:Riccd} and \eqref{def:Riccc} for the discrete and continuous Riccati operators.
First we consider the discrete algebraic Riccati equation (DARE). The following is a restatement of Theorem 17.5.3 of Lancaster and Rodman~\cite{LanRod95}
in the language of our paper.

\begin{proposition}\label{prop:DARE}
    Let $A,\Sigma\in\R^{d\times d}, H\in\R^{k\times d}, \Gamma\in\R^{k\times k}$ with $\Sigma,\Gamma\succ0$. Given a positive semi-definite matrix $C_0\in\R^{d\times d}$, iteratively define the 
    sequence $(C_n)_{n\geq 1}$ in $\R^{d\times d}_+$ by
    \begin{equation}\label{def:DARE}
    C_{n+1}= \Ricc_{d}(C_n). 
    \end{equation}
If $(H,A)$ is a $d$-detectable pair, there is a unique $C_\infty\succ 0$ independent of $C_0$ 
satisfying the stationary equation
\[ C_\infty=\Ricc_{d}(C_\infty).\]
Furthermore $C_n\rightarrow C_\infty$ as $n\rightarrow\infty$
    and  
    \begin{equation}\label{eq:rho<1}
    \rho(A-K_d(C_\infty)H)<1,
    \end{equation}
    where $K_d$ is the discrete Kalman gain
    \begin{equation} \label{def:Kd}
        K_d(C)= ACH^T(HCH^T+\Gamma)^{-1}.
    \end{equation}
\end{proposition}

We turn to the continuous algebraic Riccati equation (CARE). We state an adaptation of~\cite[Theorem 9.1.2]{LanRod95} to our setting, which is an analogue of Proposition~\ref{prop:DARE} for CARE. In addition, we state the exponential rates of convergence to the steady state~\eqref{eq:CARE;rate} due to Ocone and Pardoux~\cite[Equation (16) and Remark 2.1]{OconePardeoux96asymptotic}, which will be useful for studying the asymptotic stability of MFEnKF in Section~\ref{ssec:EnKF;cts}.
\begin{proposition}\label{prop:CARE}
Let $\sF,\sSig\in\R^{d\times d}$, $\sH\in\R^{k\times d}$, $\sG\in\R^{k\times k}$ with $\sSig,\sG\succ 0$. Let $C_0\in\R^{d\times d}_+$ and let $(C_t)_{t\geq 0}$ be the solution to the matrix differential equation 
\begin{equation}\label{def:CARE}
    \dot C_t = \Ricc_{c}(C_t). 
\end{equation}
If $(\sH,\sF)$ is a $c$-detectable pair, there is a unique $C_\infty \succ 0$ independent of $C_0$ 
satisfying the stationary equation
\[0= \Ricc_{c}(C_\infty).\] 
Furthermore $C_t\rightarrow C_\infty$ as $t\rightarrow\infty$
and 
\begin{equation}\label{eq:alpha<0}
    \alpha(\sF-K_c(C_\infty)\sH)<0,
\end{equation}
where $K_c$ is the continuous-time Kalman gain
\begin{equation}\label{def:Kc}
    K_c(C)=C\sH^T\sG^{-1}.
\end{equation}
Moreover, for any $\beta$ such that $0<\beta<-\alpha(\sF-K_c(C_\infty)\sH)$, any two solutions $(C_t^1)_{t\geq 0},(C_t^2)_{t\geq 0}$, and any choice of matrix norm, there exists a constant $\tilde c=\tilde c(C_0^1,C_0^2,\beta)$ (that can also depend on the matrix norm) such that
\begin{equation}\label{eq:CARE;rate}
    \norm{C_t^1-C_t^2}\leq \tilde ce^{-\beta t} \text{ for all }t\geq 0.
\end{equation}
\end{proposition}

\begin{remark}\label{rmk:DARECARE}
Regarding \eqref{eq:CARE;rate} due to Ocone and Pardoux~\cite{OconePardeoux96asymptotic}, while the standard observation noise $\sG=I_k$ is assumed in their paper, the results apply to cases with $\sG\succ 0$ by change of coordinates; note that all the notions provided in Definition~\ref{def:detectability} are independent of the coordinate system.
\end{remark}

When the discrete-time model is a time discretization of the continuous-time model, as discussed in Section~\ref{ssec:disc-cts}, the notions of $c$- and $d$-detectability should be consistent. We conclude with a remark on this matter.
\begin{remark}\label{rmk:disc-cts}
Introducing the time step $\tau>0$, consider the scalings \eqref{def:tau_scale}. Then the discrete Kalman gain $K_d$ \eqref{def:Kd} becomes
\[K_d(C)=(I+\tau \sF) C \sH^T(\tau \sH C\sH^T+\sG)^{-1}=K_c(C)+O(\tau)\]
where $K_c$ is the continuous time Kalman gain~\eqref{def:Kc}. Then
\[A-K_d(C)H=I+\tau (\sF-K_d(C)\sH)=I+\tau(\sF-K_c(C)\sH)+O(\tau^2).\]
Thus, as $\tau\rightarrow 0$,
\[\rho(A-K_d(C)H)=\rho(I+\tau(\sF-K_c(C)\sH)+O(\tau^2))=1+\tau\alpha(\sF-K_c(C)\sH)+o(\tau).\]
\end{remark}

\subsection{Discrete-Time MFEnKF}\label{ssec:disc_time}
Suppose the signal and observation dynamics are given by~\eqref{def:sig-obs;discrete}
\begin{equation*}
\begin{split}
    X_{n+1}^\dag &= AX_n^\dag + \xi_{n+1}^\dag \text{ where } \xi_{n+1}^\dag\sim N(0,\Sigma),\\
    Y_{n+1}^\dag &= HX_{n+1}^\dag+\eta_{n+1}^\dag \text{ where } \eta_{n+1}^\dag\sim N(0,\Gamma).
\end{split}
\end{equation*}
Recall that the corresponding MFEnKF is given by~\eqref{def:MFEnKF;disc}
\begin{align*}
    \hat X_{n+1} &= AX_n+\xi_{n+1},\\
    X_{n+1} &=\hat X_{n+1}-\Cov(\hat X_{n+1}) H^T(H\Cov(\hat X_{n+1}) H^T+\Gamma)^{-1}(H\hat X_{n+1}+\eta_{n+1}-y_{n+1}^\dag);
\end{align*}
in this section we consider the observation sequence to be deterministic, which means the MFEnKF predictive distribution $\hat\pi_{n+1}$ is simply the law (rather than the law conditioned on observations) of $\hat X_{n+1}$.
As the updates involve covariances of the random vectors $\{\hat X_{n}\}_{n\geq 1}$ but not those of $\{X_{n}\}_{n\geq 0}$, it is more convenient to take $\hat X_1$ as the starting point of the analysis. Given $X_0\sim\mu$, $\hat X_1^\mu=AX_0+\xi_1$. We note that, given $\Sigma\succ 0$ (which is part of the assumption~\eqref{ass:nondegen;d}), we have $\Cov(\hat X_1^\mu)\succ 0$, and nondegeneracy of the covariance persists for all $\{\hat X_n\}_{n\geq 1}$; see Lemma~\ref{lem:lyapunov;matrices}.

Direct calculations from \eqref{def:MFEnKF;disc} verify that the sequence $(\hat X_n)_{n\geq 1}$ satisfies the recursion
\begin{equation}\label{def:MFEnKF;disc;hatX}
\begin{split}
    \hat X_{n+1}&= A \hat X_n+\xi_{n+1}-K_d(\Cov(\hat X_n))(H\hat X_{n}+\eta_{n}-y_{n}^\dag)
\end{split}
\end{equation} 
where $K_d:\R^{d\times d}_+\rightarrow \R^{d\times k}$ is the discrete Kalman gain given by \eqref{def:Kd},
\begin{equation*}
        K_d(C)= ACH^T(HCH^T+\Gamma)^{-1}.
\end{equation*}
Furthermore, the covariance $\hat C_n=\Cov(\hat X_n)$ then satisfies the recursion
\begin{equation}\label{eq:hatC;update;K}
    \hat C_{n+1}=(A-K_d(\hat C_n)H)\hat C_n(A- K_d(\hat C_n)H)^T + \Sigma+K_d(\hat C_n)\Gamma K_d(\hat C_n)^T;
\end{equation}
as
\begin{equation}\label{eq:Riccd;Kd}
    \Ricc_d(C)=(A-K_d(C)H) C(A-K_d(C)H)^T+\Sigma+K_d(C)\Gamma K_d(C)^T,
\end{equation}
\eqref{eq:hatC;update;K} is precisely \eqref{eq:Riccd;recursion}.

Note that \eqref{eq:hatC;update;K} is a closed equation for evolution
of covariances. In particular, if we denote by $(\hat C_{n}^\mu)_{n\geq 0},(\hat C_{n}^\nu)_{n\geq 0}$ the sequence of covariance matrices with initial probability measure $\mu,\nu\in\P_2(\R^d)$ with the identical covariance, by Proposition~\ref{prop:Riccati-d} we deduce
\begin{equation}\label{eq:cov;propagate}
\hat C_n^\mu=\hat C_n^{\nu} \text{ for all } n\geq 1 \text{ if } \Cov(\mu)=\Cov(\nu)
\end{equation}
even if $\mu\neq \nu$; this crucially relies on the linearity of the signal-observation dynamics~\eqref{def:sig-obs;discrete}.

We can leverage \eqref{eq:cov;propagate} to quantify how fast the mean-field EnKF directs probability measures to the Gaussian manifold when the Kalman filter itself is stable. To this end, we compare the evolution of the two solutions of the mean-field EnKF with initial data $\nu$ and its Gaussian projection $\G\nu$, where the Gaussian projection is defined in \eqref{def:G}. Recall that $\G\nu$ shares the mean and the covariance with $\nu$.

\begin{proposition}\label{prop:gauss_conv;W2;asymp}
    For each $\nu\in\P_2(\R^d)$, denote by $\hat\pi_n^\nu$ the law of $\hat X_n^\nu$ defined by~\eqref{def:MFEnKF;disc} with initial datum $X_0^\nu\sim \nu\in\P_2(\R^d)$. Suppose Assumption~\ref{ass:disc} holds. Then, for any $\mu\in\P_2(\R^d)$, $\hat C_n^\mu\rightarrow \hat C_\infty$, with limiting covariance satisfying
    \[\rho(A-K_d(\hat C_\infty)H)<1,\] 
    and
    \begin{equation}\label{eq:gauss_conv;disc;asymp}
        \limsup_{n\rightarrow\infty}\left(\frac{W_2(\hat\pi_n^\mu,\hat\pi_n^{\G\mu})}{W_2(\mu,\G\mu)}\right)^{\frac{1}{n}}
        \leq \rho(A-K_d(\hat C_\infty)H),
    \end{equation}
   where $K_d(C)$ is the discrete-time Kalman gain matrix defined in \eqref{def:Kd} and we use the convention $0/0=0$.

\end{proposition}
\begin{proof} 
That $\hat C_n^\mu\rightarrow\hat C_\infty$ with $\rho(A-K_d(\hat C_\infty)H)<1$ is an immediate consequence of Proposition~\ref{prop:DARE}, which is classical in the theory of the Kalman filter. The result is applicable here as the covariance evolution equations are identical even when the initial laws are non-Gaussian.

We use a coupling argument to show \eqref{eq:gauss_conv;disc;asymp} by inductively defining the synchronous coupling. To this end, first note $\Cov(\mu)=\Cov(\G\mu)$ by construction, we know that the sequences of covariance matrices $(\hat C_{n}^\mu)_{n\geq 1},(\hat C_n^{\G\mu})_{n\geq 1}$ coincide. Consider the sequences of random vectors $(\hat X_n^\mu)_{n}$ and $(\hat X_n^{\G\mu})_{n}$ defined by \eqref{def:MFEnKF;disc} with initial law $X_0^\mu\sim\mu$ and $X_0^{\G\mu}\sim\G\mu$. As the covariance matrices coincide, we deduce from \eqref{def:MFEnKF;disc;hatX} that
    \begin{equation}\label{eq:hatx_diff}
    \hat X_{n}^\mu-\hat X_{n}^{\G\mu}=(A-K_d(\hat C_{n-1}^\mu)H)(\hat X_{n-1}^\mu-\hat X_{n-1}^{\G\mu})-K_d(\hat C_{n-1})(\eta_{n-1}^\mu-\eta_{n-1}^{\G\mu})+\xi_{n}^\mu-\xi_{n}^{\G\mu} \text{ for each } n\geq 2.
    \end{equation}
    
We now describe the coupling $\hat \gamma_n$ of $\hat X_n^\mu$ and $\hat X_n^{\G\mu}$ inductively. For the base case, choose the optimal coupling $\hat\gamma_1$ of $\hat X_{1}^\mu$,$\hat X_1^{\G\mu}$, namely
    \[\E^{(\hat X_1^\mu,\hat X_{1}^{\G\mu})\sim\hat\gamma_1}|\hat X_1^{\mu}-\hat X_1^{\G\mu}|^2=W_2^2(\hat\pi_1^\mu,\hat\pi_1^{\G\mu}).\]
    Suppose we are given $\hat\gamma_{n-1}$. As $\xi_{n}^\mu,\xi_{n}^{\G\mu},\eta_{n-1}^\mu,\eta_{n-1}^{\G\mu}$ are independent from each other and all earlier sources of randomness hence from $\hat\gamma_{n-1}$, we can extend $\hat\gamma_{n-1}$ to construct $\hat\gamma_n$ under which $\xi_{n}^\mu=\xi_n^{\G\mu}$ and $\eta_{n-1}^\mu=\eta_{n-1}^{\G\mu}$ -- i.e. under $\hat\gamma_n$,
    \begin{align*}
        (\hat X_{n-1}^\mu,\hat X_{n-1}^{\G\mu})\sim\hat\gamma_{n-1},\; \eta_{n-1}^\mu=\eta_{n-1}^{\G\mu},\; \xi_{n}^{\mu}=\xi_{n}^{\G\mu}.
    \end{align*}
    Then, all the randomness in \eqref{eq:hatx_diff} cancels out under this coupling, and thus
    \begin{align*}
        W_2^2(\hat\pi_{n}^\mu,\hat\pi_{n}^{\G\mu})&\leq \E^{(\hat X_{n}^\mu,\hat X_{n}^{\G\mu})\sim\hat\gamma_{n}}|\hat X_{n}^{\mu}-\hat X_{n}^{\G\mu}|^2
        =\E^{\hat\gamma_{n}}|(A-K_d(\hat C_{n-1}^\mu)H)(\hat X_{n-1}^{\mu}-\hat X_{n-1}^{\G\mu})|^2 \\
        &=\E^{\hat \gamma_{n}} \left|\left(\prod_{j=1}^{n-1}(A-K_d(\hat C_j^\mu)H)\right)(\hat X_1^\mu-\hat X_1^{\G\mu})\right|^2 \\
        &\leq\norm{\prod_{j=1}^{n-1}(A-K_d(\hat C_j^\mu)H)}_2^2\E^{\hat\gamma_{n}}|\hat X_1^\mu-\hat X_1^{\G\mu}|^2\\
        &= \norm{\prod_{j=1}^{n-1}(A-K_d(\hat C_j^\mu)H)}_2^2 W_2^2(\hat\pi_1^\mu,\hat\pi_1^{\G\mu}).
    \end{align*}
    Statement \eqref{eq:gauss_conv;disc;asymp} is then a consequence of the Gelfand's formula $\norm{B^n}_2^{1/n}\xrightarrow[]{n\rightarrow\infty} \rho(B)$ for any matrix $B$, a detailed argument for which we provide below for completeness.
    
    Let $B:=A-K_d(\hat C_\infty)H$, and $E_n=K_d(\hat C_\infty)H-K_d(\hat C_n^\mu)H$, noting that $E_n\rightarrow 0$ as $n\rightarrow\infty$. 
    Thus,
    \begin{align*}
        \left(\frac{W_2(\hat\pi_n^\mu,\hat\pi_n^{\G\mu})}{W_2(\mu,\G\mu)}\right)^{1/n}\leq 
        \left(\frac{W_2(\hat\pi_1^\mu,\hat\pi_1^{\G\mu})}{W_2(\mu,\G\mu)}\right)^{1/n}\norm{\prod_{j=1}^{n-1}(B+E_j)}_2^{1/n}.
    \end{align*}
    Clearly $(W_2(\hat\pi_1^\mu,\hat\pi_1^{\G\mu})/W_2(\mu,\G\mu))$ is bounded hence its $n^{th}$ root converges to 1. 
    
    To control the product term, note that for any $\eps>0$ and $m\geq 1$, there exists $N(\eps,m)$ sufficiently large such that
    \begin{equation}\label{eq:prod(B+E)}
    \norm{\prod_{j=im+1}^{(i+1)m}(B+E_j)}_2\leq \norm{B^m}_2+\eps \text{ for all } i\geq N(\eps,m).
    \end{equation}
    Fix $m\in\N$
    and for any $n\ge m+1$, write $n-1=qm+r$ where $q$ is the quotient $q=\lfloor \frac{n-1}{m}\rfloor\in\Z_+$ and $r\in\{0,1,\cdots,m-1\}$ is the remainder. Suppose $n$ is sufficiently large such that $q\geq N(\eps,m)$. Then, (for the remainder of the proof, matrix products are ordered with larger indices to the left) 
    \begin{align*}
        \prod_{j=1}^{n-1}(B+E_j)=\prod_{j=1}^{qm+r}(B+E_j)=&\underbrace{\left(\prod_{i=0}^{N-1}\prod_{j=im+1}^{(i+1)m}(B+E_j)\right)}_{:=S_1(N,m)}\cdot
        \\ &\cdot\underbrace{\left(\prod_{i=N}^{q-1}\prod_{j=im+1}^{(i+1)m}(B+E_j)\right)}_{:=S_2(N,q)}\underbrace{\left(\prod_{j=qm+1}^{n-1}(B+E_j)\right)}_{:=S_3(q).}.
    \end{align*}
    As $E_j\xrightarrow[]{j\rightarrow \infty} 0$, $R:=\sup_{j\geq 1}\|B+E_j\|_2<+\infty$. Replacing $R$ by $\max\{1,R\}$ if necessary, we deduce that $\|S_1(N,m)\|_{2}\leq R^{Nm}$, independent of $n$. Moreover, by \eqref{eq:prod(B+E)} we have
    \[\|S_2(N,q)\|_2\leq (\|B^m\|_2+\eps)^{q-N},\]
    and $\|S_3(q)\|_2\leq R^r\leq R^{m-1}$. Hence, taking $n^{th}$ root,
    \[\norm{\prod_{j=1}^{n-1}(B+E_j)}_2^{1/n}\leq R^{\frac{(N+1)m-1}{n}}(\|B^m\|_2+\eps)^{\frac{q-N}{n}}.\]
    Taking $n\rightarrow\infty$ with $\eps,m$ (hence $N(\eps,m)$ fixed), note that $\frac{q}{n}\rightarrow\frac{1}{m}$ and $R^{\frac{(N+1)m-1}{n}}\rightarrow 1$; then taking $\eps\rightarrow 0$, we yield
    \[\limsup_{n\rightarrow\infty}\norm{\prod_{j=1}^{n-1}(B+E_j)}_2^{1/n}\leq  \|B^m\|_2^{1/m}.\]
    Finally, we can let $m\rightarrow\infty$ and apply Gelfand's formula to obtain that the left-hand side is controlled by $\rho(B)$.

\end{proof}

While Proposition~\ref{prop:gauss_conv;W2;asymp} tells us that the MFEnKF exponentially contracts toward the Gaussian manifold \emph{asymptotically} as $n\rightarrow\infty$, as discussed in the beginning of the section, precise nonasymptotic bounds are desirable. As $\rho(A-K_d(\hat C_\infty)H)<1$ \emph{ does not }imply $\|A-K_d(\hat C_\infty)H\|_2<1$, without considerably strengthening the assumptions we cannot hope to deduce a strictly contractive bound from the argument in the proof of Proposition~\ref{prop:gauss_conv;W2;asymp}.

Instead, we use the covariance-weighted Wasserstein distances used in the variational derivation of the EnKF in Section~\ref{sec:EnKF;var_der} to obtain a nonasymptotic bound. To this end, we will need the following technical lemma.
\begin{lemma}\label{lem:lyapunov;matrices}
    Let $A\in\R^{d\times d}$ be an invertible matrix, and let $\Ricc_{d}$ be the discrete Riccati operator defined in~\eqref{def:Riccd}. Given $\hat C_n\in\R_{++}^{d\times d}$, we have $\Ricc_{d}(\hat C_n)\in\R_{++}^{d\times d}$, and 
    \begin{equation}\label{eq:lyapunov;matrices}
    \begin{split}
        &(A-K_d(\hat C_n) H)^T (\Ricc_{d}(\hat C_n))^{-1} (A-K_d(\hat C_n)H) \\&= \hat C_n^{-1} 
        - H^T(H \hat C_n H^T+\Gamma)^{-1}H-\hat C_n^{-1}(\hat C_n^{-1}+H^T\Gamma^{-1}H+A^T\Sigma^{-1} A)^{-1}\hat C_n^{-1}.
    \end{split}
    \end{equation}
\end{lemma}

\begin{proof}
    For simplicity, let us temporarily write
    \[K_n:= \hat C_n H^T(H \hat C_n H^T+\Gamma)^{-1}=A^{-1}K_d(\hat C_n).\]
    Then we can write the recursion $\hat C_{n+1}=\Ricc_{d}(\hat C_n)$ in two steps,
    \[C_n=(I-K_n H)\hat C_n,\quad \hat C_{n+1}=A C_n A^T+\Sigma;\]
    by the Woodbury matrix identity in the form~\eqref{eq:woodbury1}, the first update is equivalent to
    \[C_n^{-1}=\hat C_n^{-1}+H^T\Gamma^{-1}H,\]
    thus $\hat C_n\in\R_{++}^{d\times d}$ implies $C_n\in\R_{++}^{d\times d}$, and from $\Sigma\succ0$ follows that $\hat C_{n+1}$ is also invertible.
    
    We claim
    \begin{equation}\label{eq:ATCA}
        A^T\hat C_{n+1}^{-1} A= C_n^{-1}-C_n^{-1}(C_n^{-1}+A^T\Sigma^{-1}A)^{-1}C_n^{-1},
    \end{equation}
    and
    \begin{equation}\label{eq:MCM}
        (I-K_n H)^TC_n^{-1} (I-K_nH) =\hat C_{n}^{-1}(I-K_n H).
    \end{equation}
    Indeed, \eqref{eq:ATCA} follows by applying the Woodbury matrix identity~\eqref{eq:woodbury} three times; we use the identity with the different notation
    \[(W+UZV)^{-1}=W^{-1}-W^{-1}U(Z^{-1}+V W^{-1} U)^{-1}VW^{-1}\]
    to avoid conflict with notation in this proof. First, apply the above with $W=\Sigma$, $U=A$, $V=A^T$, $Z=C_n$ and then with $W=A^{-1}\Sigma A^{-T}$, $U=V=I$, and $Z=C_n$, we deduce
    \begin{align*}
        A^T\hat C_{n+1}^{-1}A&=A^T(AC_n A^T+\Sigma)^{-1} A
        = A^T\Sigma^{-1}A-A^T \Sigma^{-1} A(C_n^{-1}+A^T\Sigma^{-1}A)^{-1}A^T\Sigma^{-1}A  \\
        &=(A^{-1}\Sigma A^{-T}+C_n)^{-1}.
    \end{align*}
    Finally, applying the Woodbury matrix identity with $U=V=I$, $Z=A^{-1}\Sigma A^{-T}$, and $W=C_n$, we conclude \eqref{eq:ATCA}. To see \eqref{eq:MCM}, we use $C_n=(I-K_n H)\hat C_n$ and note
    \begin{align*}
        (I-K_n H)^T C_n^{-1} (I-K_n H) = (I-K_n H)^T \hat C_{n}^{-1},
    \end{align*}
    and we can take the transpose to obtain \eqref{eq:MCM}.

    Thus, recalling that $K_n=A^{-1}K_d(\hat C_n)$,
    \begin{align*}
        &(A-K_d(\hat C_n) H)^T (\Ricc_{d}(\hat C_n))^{-1} (A-K_d(\hat C_n)H)=(I-K_n H)^T A^T \hat C_{n+1}^{-1} A (I-K_n H)\\
        &=(I-K_n H)^T C_n^{-1} (I-K_n H)
        -(I-K_n H)^T C_n^{-1}(C_n^{-1}+A^T\Sigma^{-1}A)^{-1} C_n^{-1} (I-K_n H)\\
        &= \hat C_n^{-1} (I-K_n H)- \hat C_n^{-1}(C_n^{-1}+A^T\Sigma^{-1} A)^{-1}\hat C_n^{-1}\\
        &=\hat C_n^{-1}- H^T(H\hat C_n H^T+\Gamma)^{-1} H - \hat C_n^{-1}( C_n^{-1}+A^T\Sigma^{-1} A)^{-1}\hat C_n^{-1};
    \end{align*}
    we have used \eqref{eq:ATCA} for the second equality, \eqref{eq:MCM} and the definition of $C_n$ in the third, and expanded the terms in the fourth. Finally, applying $C_n^{-1}=\hat C_n^{-1}+H^T\Gamma^{-1}H$ to the last expression, we deduce \eqref{eq:lyapunov;matrices}.
\end{proof}

\begin{theorem}\label{thm:gauss_lyapunov;disc}
Suppose $A,\Sigma\in\R^{d\times d}$, $H\in\R^{k\times d}$, $\Gamma\in\R^{k\times k}$ satisfy Assumption~\ref{ass:disc}. Let $\mu,\nu\in\P_2(\R^d)$ with $\Cov(\mu)=\Cov(\nu)$, and let $\hat\pi_n^\nu,\hat\pi_n^\mu$ be the law of $\hat X_n^\nu,\hat X_n^\mu$ defined by~\eqref{def:MFEnKF;disc} with initial data $X_0^\nu\sim \nu,X_0^\mu\sim\mu$. Recall that $\hat C_{n}^\mu:=\Cov(\hat\pi_{n}^\mu)=\Cov(\hat\pi_n^\nu)$ and they are independent of the observation sequence $(y_n^\dag)_{n\geq 1}$.

We have
\begin{equation}\label{eq:gauss_conv;Lyapunov;disc}
    W_2^2(\hat\pi_{n+1}^\mu,\hat\pi_{n+1}^\nu;\hat C_{n+1}^\mu)\leq (1-\beta_d(\hat C_n))W_2^2(\hat\pi_n^\mu,\hat\pi_n^\nu;\hat C_{n}^\mu)
\end{equation}
where $\beta_d:\R^{d\times d}_{++}\rightarrow \R_{++}$ is given by
\begin{equation}\label{def:beta}
    \beta_d(\hat C)=\lambda_{\min}\left(\hat C^{1/2}H^T(H\hat C H^T+\Gamma)^{-1}H\hat C^{1/2}+\hat C^{-1/2}(\hat C^{-1}+H^T\Gamma^{-1}H+A^T\Sigma^{-1}A)^{-1}\hat C^{-1/2}\right).
\end{equation}
In particular, at the stable covariance matrix $\hat C_\infty\succ 0$, it simplifies to
\begin{equation}\label{eq:betan;Cinfty}
    \beta_d(\hat C_\infty)=\lambda_{\min}(\hat C_\infty^{-1/2}(\Sigma+K_d(\hat C_\infty)\Gamma K_d(\hat C_\infty)^T)\hat C_\infty^{-1/2}).
\end{equation}
Moreover, for all $n \in \mathbb{N},$ $\beta_d(\hat C_n)\geq \beta_d^0>0$ for some $\beta_d^0$ that may depend on $\Cov(\mu)$, and
\begin{equation}\label{eq:gauss_conv;disc;exp}
    W_2^2(\hat\pi_{n+1}^\mu,\hat\pi_{n+1}^\nu;\hat C_{n+1}^\mu)\leq (1-\beta_d^0)^n W_2^2(\hat\pi_1^\mu,\hat\pi_1^\nu; \hat C_{1}^\mu).
\end{equation}

\end{theorem}

\begin{proof}
As in the proof of Proposition~\ref{prop:gauss_conv;W2;asymp}, we use a coupling argument.
Consider the sequences of random vectors $(\hat X_n^\mu)_{n}$ and $(\hat X_n^{\nu})_{n}$ solving \eqref{def:MFEnKF;disc} with initial law $X_0^\mu\sim \mu$ and $X_0^{\nu}\sim\nu$ respectively.
As $\Cov(\mu)=\Cov(\nu)$, the sequences of covariance matrices $\hat C_{n}^\mu=\hat C_n^{\nu}$ for all $n\geq 1$. 
Thus, by \eqref{def:MFEnKF;disc;hatX},
    \[
    \hat X_{n+1}^\mu-\hat X_{n+1}^{\nu}=(A-K_d(\hat C_{n}^\mu)H)(\hat X_{n}^\mu-\hat X_{n}^{\nu})-K_d(\hat C_{n})(\eta_{n}^\mu-\eta_{n}^{\nu})+\xi_{n+1}^\mu-\xi_{n+1}^{\nu}.
    \]
    Let $\hat\gamma_n$ 
    be the optimal coupling between $\hat X_n^\mu\sim\hat\pi_n^\mu$ and $\hat X_n^\nu\sim\hat\pi_n^\nu$ with respect to $W_{2,\hat C_n^\mu}$ -- i.e.
    \[\E^{(\hat X_n^\mu,\hat X_{n}^{\nu})}|\hat X_n^{\mu}-\hat X_n^{\nu}|_{\hat C_n^\mu}^2=W_2^2(\hat\pi_n^\mu,\hat\pi_n^{\nu};\hat C_n^\mu).\]
    As $\eta_{n}^\mu,\eta_{n}^{\nu},\xi_{n+1}^\mu,\xi_{n+1}^{\nu}$, are independent from each other and $\hat X_{n}^\mu,\hat X_{n}^\nu$ (hence from  $\hat\gamma_n$), we can construct a coupling $\hat\gamma_{n+1}$ of $\hat\pi_{n+1}^\mu,\hat\pi_{n+1}^\nu$ under which $\eta_{n}^\mu=\eta_{n}^\nu$, $\xi_{n+1}^\mu=\xi_{n+1}^\nu$, and that has the marginals $(\hat X_n^\mu,\hat X_n^{\nu})\sim\hat\gamma_n$.
    Under this coupling, all the randomness cancels out and
    \[\hat X_{n+1}^\mu-\hat X_{n+1}^\nu = (A-K_d(\hat C_n)H)(\hat X_n^\mu-\hat X_n^\nu).\]
    Moreover, thanks to the nondegeneracy of $\Sigma$, $\hat C_1\succ 0$, we have by Lemma~\ref{lem:lyapunov;matrices} that $\hat C_n\succ 0$ for all $n\geq 1$. Thus we can apply \eqref{eq:lyapunov;matrices} to obtain
    \begin{align*}
        &W_2^2(\hat\pi_{n+1}^\mu,\hat\pi_{n+1}^\nu;\hat C_{n+1})
        \leq \E^{(\hat X_{n+1}^\mu,\hat X_{n+1}^\nu)\sim\hat\gamma_{n+1}}|\hat X_{n+1}^\mu-\hat X_{n+1}^{\nu}|_{\hat C_{n+1}}^2 \\
        &=\E^{(\hat X_{n}^\mu,\hat X_{n}^\nu)\sim\hat\gamma_n} (\hat X_n^\mu-\hat X_n^\nu)^T(A-K_d(\hat C_{n})H)^T \Ricc_{d}(\hat C_n)^{-1} (A-K_d(\hat C_n)H)(\hat X_n^\mu-\hat X_n^\nu)    \\
        &=\E^{(\hat X_{n}^\mu,\hat X_{n}^\nu)\sim\hat\gamma_n} (\hat X_n^\mu-\hat X_n^\nu)^T(\hat C_n^{-1}-M_n) (\hat X_n^\mu-\hat X_n^\nu),
    \end{align*}
    where
    \begin{equation}\label{def:Mn}
        M_n:= H^T(H \hat C_n H^T+\Gamma)^{-1}H+\hat C_n^{-1}(\hat C_n^{-1}+H^T\Gamma^{-1}H+A^T\Sigma^{-1} A)^{-1}\hat C_n^{-1}.
    \end{equation}
    Note 
    \begin{equation}\label{eq:Mgeqbeta}
        \hat C_{n}^{1/2}M_n\hat C_n^{1/2}\succeq \beta I 
    \end{equation}
    for some $\beta>0$ if and only if $M_n\succeq \beta \hat C_n^{-1}$. Thus, by the definition \eqref{def:beta} of $\beta_d:\R_{++}^{d\times d}\rightarrow\R_{++}$, we obtain
    \begin{align*}
        W_2^2(\hat\pi_{n+1}^\mu,\hat\pi_{n+1}^\nu;\hat C_{n+1})
        &\leq (1-\beta_d(\hat C_n))\E^{(\hat X_{n}^\mu,\hat X_{n}^\nu)\sim\hat\gamma_n} (\hat X_n^\mu-\hat X_n^\nu)^T\hat C_n^{-1} (\hat X_n^\mu-\hat X_n^\nu)\\
        &=(1-\beta_d(\hat C_n))W_2^2(\hat\pi_n^\mu,\hat\pi_n^\nu;\hat C_n).
    \end{align*}
    Clearly $\hat C\mapsto \beta_d(\hat C)$ is continuous. As the detectability condition and $\sSig\succ 0$ guarantee that $\hat C_n\rightarrow \hat C_\infty \in\R^{d\times d}_{++}$, we have that $\beta_d(\hat C_n^\mu)$ is bounded away from zero uniformly in $n$. Thus, \eqref{eq:gauss_conv;disc;exp} follows.
    \medskip
    
    It remains to show \eqref{eq:betan;Cinfty}. Let us write 
    \begin{equation}\label{def:Phi_infty}
    \Phi_\infty:=(A-K_d(\hat C_\infty)H)\in\R^{d\times d}.
    \end{equation}
    As $\Ricc_{d}(\hat C_\infty)=\hat C_\infty$, specializing \eqref{eq:lyapunov;matrices} to $\hat C_\infty$, we obtain 
    \[\Phi_\infty^T\hat C_\infty^{-1}\Phi_\infty=\hat C_\infty^{-1}-M_\infty,\]
    where $M_\infty$ is defined by \eqref{def:Mn} with $\hat C_\infty$ in place of $\hat C_n$. Combining with \eqref{eq:Mgeqbeta} for $M_\infty,\hat C_\infty$, $\beta_\infty:=\beta_d(\hat C_\infty)$ is the largest constant $\beta>0$ satisfying
    \[\beta I\preceq\hat C_\infty^{1/2}M_\infty\hat C_\infty^{1/2}=I-\hat C_\infty^{1/2}\Phi_\infty^T \hat C_\infty^{-1}\Phi_\infty\hat C_\infty^{1/2}.\]
    Thus, $\beta_\infty$ can be equivalently characterized as the largest $\beta>0$ satisfying
    \[\hat C_\infty^{1/2}\Phi_\infty^T\hat C_\infty^{-1} \Phi_\infty\hat C_\infty^{1/2}\preceq (1-\beta)I.\]
    In turn, this is equivalent to
    \[\lambda_{\max}[(\hat C_\infty^{-1/2}\Phi_\infty \hat C_\infty^{1/2})(\hat C_\infty^{-1/2}\Phi_\infty \hat C_\infty^{1/2})^T]=\lambda_{\max}[(\hat C_\infty^{-1/2}\Phi_\infty \hat C_\infty^{1/2})^T(\hat C_\infty^{-1/2}\Phi_\infty \hat C_\infty^{1/2})]= (1-\beta_\infty).\]
    Using \eqref{eq:Riccd;Kd} and that $\hat C_\infty=\Ricc_{d}(\hat C_\infty)$,
    \begin{align*}
        \hat C_\infty = \Phi_\infty \hat C_\infty \Phi_\infty^T + \Sigma+K_d(\hat C_\infty)\Gamma K_d(\hat C_\infty)^T.
    \end{align*}
    By multiplying on the left and right by $\hat C_\infty^{-1/2}$,
    \[I=(\hat C_\infty^{-1/2}\Phi_\infty \hat C_\infty^{1/2})(\hat C_\infty^{-1/2}\Phi_\infty \hat C_\infty^{1/2})^T+\hat C_\infty^{-1/2}(\Sigma+K_d(\hat  C_\infty)\Gamma K_d(\hat C_\infty)^T)\hat C_\infty^{-1/2}.\]
    Thus
    \[\beta_\infty=\lambda_{\min}(\hat C_\infty^{-1/2}(\Sigma+K_d(\hat  C_\infty)\Gamma K_d(\hat C_\infty)^T)\hat C_\infty^{-1/2}).\]
\end{proof}

\begin{remark}\label{rmk:eq:gauss_conv;lyapunov;disc}
    We note that the signal noise $\xi_n$ and observation noise $\eta_n$ need not be Gaussian for Theorem~\ref{thm:gauss_lyapunov;disc} to hold. Indeed, the proof relies primarily on the fact that the covariance matrices satisfy the Riccati recursion~\eqref{eq:Riccd;recursion}, which is true under broader generality; see Proposition~\ref{prop:Riccati-d}. However, the result is more meaningful in the linear setting with Gaussian noise, where the Gaussian manifold is invariant under the MFEnKF, hence \eqref{eq:gauss_conv;Lyapunov;disc} with $\nu=\G\mu$ leads to contraction of the (predictive) MFEnKF distributions towards the Gaussian manifold~\eqref{sum:gauss_conv;disc}.
\end{remark}

Let us comment on the relationship between the MFEnKF distributions $(\pi_n^\mu)_{n\in\N}$~\eqref{eq:MFEnKF;algo} with the optimal filtering distributions $(\bar\pi_n^\nu)_{n\in\N}$~\eqref{def:barpi_n}.
We have shown that $(\hat\pi_n^\mu)_{n\in\N}$ exponentially contracts towards the Gaussian manifold, where the MFEnKF distributions coincide with the optimal filtering distributions. Assume that, with respect to the initial data, the optimal filtering distribution $\bar\pi_n^\nu$ is a.s. asymptotically stable. Then we can expect, by use of a triangle
inequality argument, that a result of the following type will hold:
\[\pi_n^\mu-\bar\pi_n^\nu\xrightarrow[]{n\rightarrow\infty} 0 \text{ for almost every sample path of } (Y_j^\dag)_{j\in\N}.\]
Unfortunately, however, we could not locate in the literature such an almost-sure asymptotic stability of the optimal filter in the discrete-time setting under detectability. However, we will be able to deduce an analogous result in the continuous-time setting,
the subject of the next subsection;
see Corollary~\ref{cor:MFEnKF_accurate;cts}.

\subsection{Continuous-Time MFEnKF}\label{ssec:EnKF;cts}

In this section we consider the continuous-time signal and observation model~\eqref{eq:linearSigObs;cts}
\begin{equation*}
\begin{split}
    dX_t^\dag &= \sF X_t^\dag\;dt +\sqrt{\sans{\Sigma}}\;dB_t^\dag,\\
    dZ_t^\dag &= \sH X_t^\dag\;dt + \sqrt{\sG}\;dW_t^\dag;
\end{split}
\end{equation*}
see Section~\ref{ssec:disc-cts} for the connection to the discrete-time model.
Recall that the corresponding continuous-time MFEnKF is given by~\eqref{eq:EnKF;mf;cts}
\begin{equation*}
    d X_t=\sF X_t\;dt + \sqrt{\sans{\Sigma}}\;dB_t - \Cov(\pi_t)\sH^T\sG^{-1}(\sH X_t\;dt+\sqrt{\sG}dW_t- dZ_t^\dag),
\end{equation*}
where $\pi_t$ is the conditional law of $X_t$ given $(Z_s^\dag)_{s\leq t}$, which also depends on the initial law of $X_0$. For this reason, we use the superscript $\dag$  to separate the observation $(Z_t^\dag)_t$ from other sources of randomness, the two independent copies of the Brownian motion $(B_t)_{t\geq 0}$, $(W_t)_{t\geq 0}$ which are also independent from the observation.

Thus, for each $t>0$ the measure $\pi_t$ is a random measure, and
we are interested in its pathwise almost sure behavior -- i.e, for almost every sample path of $(Z_t^\dag)_{t\geq 0}$ how does $(\pi_t)_{t\geq 0}$ behave depending on its initial datum? This is motivated by the fact that in practice we are provided with a single realization of the observation path, hence the randomness associated to it cannot be `averaged out', as for the Brownian noise $(B_t)_{t\geq 0},(W_t)_{t\geq 0}$ which are (approximately) generated by the algorithms.
The concluding theory at the end of this section quantifies
the accuracy of \eqref{eq:EnKF;mf;cts}: we compare the MFEnKF distributions 
$(\pi_t^\mu)_{t\geq 0}$ with the optimal filtering distributions $(\bar\pi_t^\nu)_{t\geq 0}$.

Henceforth, we will overload the notation for the continuous Kalman gain~\eqref{def:Kc} and write
\begin{equation}\label{def:Kc;pi}
    K_c(\pi):=\Cov(\pi)\sH^T\sG^{-1};
\end{equation}
recall that by Proposition~\ref{prop:Riccati-c}, the covariance of $\pi_t^\mu$ is independent of observations, hence so is the Kalman gain.

We begin with the continuous-time analogue of Theorem~\ref{thm:gauss_lyapunov;disc}.

\begin{theorem}\label{thm:gauss_lyapunov;cts}
    Suppose $\sF,\sSig\in\R^{d\times d}$, $\sH\in\R^{k\times d}$, $\sG\in\R^{k\times k}$ satisfy Assumption~\ref{ass:cts}. Given $\nu\in\P_2(\R^d)$, let $(X_t^\nu)_{t\geq 0}$ be the random vector defined by \eqref{eq:EnKF;mf;cts} with the initial datum $X_0\sim\nu$, and $\pi_t^\nu$ its conditional law given $(Z_t^\dag)_{t\geq 0}$.
    Recall that $\Cov(\pi_t^\nu)$ is independent of $(Z_t^\dag)_{t\geq 0}$.
    For any $\mu\in\P_2(\R^d)$ with $\Cov(\mu)\succ 0$,
    \begin{equation}\label{eq:gauss_conv;lin_cts}
    W_2(\pi_t^{\mu},\pi_t^{\G\mu};\Cov(\pi_t^\mu))
    \leq \exp\left(-\frac12\int_0^t \beta_c(\Cov(\pi_s^\mu))\;ds\right)W_2(\mu,\G\mu;\Cov(\mu))
    \end{equation}
    for every sample path of $(Z_t^\dag)_{t\geq 0}$, where $\beta_c:\R_{++}^{d\times d}\rightarrow\R_{++}$ is defined by
    \begin{equation}\label{def:beta_t}
        \beta_c(C):=\lambda_{\min}\left(C^{1/2}\sH^T\sG^{-1}\sH C^{1/2}+C^{-1/2}\sSig C^{-1/2}\right).
    \end{equation}
    Moreover, for all $t \ge 0$, $\beta_c(\Cov(\pi_t^\mu))\geq \beta_0$ for some $\beta_0>0$ depending on $\Cov(\mu)$, and thus
    \begin{equation}\label{eq:gauss_conv;;lin_cts;beta0}
        W_2(\pi_t^\mu,\pi_t^{\G\mu};\Cov(\pi_t^\mu))\leq e^{-\frac12\beta_0 t}W_2(\mu,\G\mu;\Cov(\mu)) \text{ for all } t\geq 0. 
    \end{equation}
\end{theorem}
\begin{proof}
Let $(X_t^\mu)_{t\geq 0},(X_t^{\G\mu})_{t\geq 0}$ be two solutions of \eqref{eq:EnKF;mf;cts} with initial data $\mu,\G\mu$ respectively. Then the covariance of the two random vectors are identical for $t\geq 0$. 

We construct the following coupling $\gamma_t$ between $\pi_t^\mu,\pi_t^{\G\mu}$: conditioned on the observation path driving both systems, $X_0^\mu,X_0^{\G\mu}$ are distributed according to the $W_2(\cdot,\cdot;\Cov(\mu))$-optimal coupling ($\gamma_0$, say) of $\mu,\G\mu$, and the Brownian motions $(B_s)_{s\leq t}, (W_s)_{s\leq t}$ (which are independent from each other and the initial laws), are coupled synchronously by $\gamma_t$. Then, under this joint law, $\gamma_t$ of $(X_t^\mu,X_t^{\G\mu})$,
\begin{align*}
    d(X_t^\mu-X_t^{\G\mu})= (\sF-\Cov(\pi_t^\mu)\sH^T\sG^{-1}\sH)(X_t^\mu-X_t^{\G\mu})\;dt \text{ for all } t>0.
\end{align*}
In particular, all the randomness cancels. For simplicity, write $D_t:= X_t^\mu-X_t^{\G\mu}$, $C_t:=\Cov(\pi_t^\mu)$. Given $\Cov(\mu)\succ 0$, it is straightforward to verify that $C_t\succ 0$ for all $t\geq 0$, for instance from the representation of $C_t$ by the variation of constants formula. Moreover, direct calculations reveal that $P_t:=C_t^{-1}$ satisfies the differential equation
\[\dot P_t = -P_t \sF- \sF^T P_t - P_t\sans{\Sigma} P_t+\sH^T\sG^{-1}\sH=:R(P_t).\]
As $C_t\rightarrow C_\infty\succ 0$ by Proposition~\ref{prop:CARE}, $P_t\rightarrow P_\infty$ where $P_\infty=C_\infty^{-1}\succ 0$ is the unique positive definite matrix satisfying $R(P_\infty)=0$.
Thus, under the coupling $\gamma_t$
\begin{align*}
    \frac{d}{dt}|D_t|_{C_t}^2&=\dot D_t^T P_t D_t+ D_t^T P_t \dot D_t + D_t^T \dot P_t D_t   \\
    &=D_t^T(\sF-K_c(C_t)\sH)^T P_t D_t+D_t^T P_t(\sF-K_c(C_t)\sH) D_t + D_t^T R(P_t) D_t.
\end{align*}
Defining $\beta_t$ by
\[\beta_t:=-\lambda_{\max}\left[ P_t^{-1/2}(\sF-K_c(C_t)\sH)^T P_t^{1/2} + P_t^{1/2}(\sF-K_c(C_t)\sH)P_t^{-1/2}+ P_t^{-1/2} R(P_t) P_t^{-1/2}\right],\]
we have 
\begin{align*}
     \frac{d}{dt}|D_t|_{C_t}^2\leq -\beta_t |D_t|_{C_t}^2.
\end{align*}
Thus, applying the Gronwall inequality and taking expectations over the coupling,
\begin{align*}
    W_2^2(\pi_t^\mu,\pi_t^{\G\mu};C_t)&\leq \E^{(X_t^\mu,X_t^{\G\mu})\sim\gamma_t}|X_t^\mu-X_t^{\G\mu}|_{C_t}^2\\
    &\leq \exp\left(-\int_0^t \beta_s\;ds\right)\E^{(X_0^\mu,X_0^{\G\mu})\sim\gamma_0}|X_0^\mu-X_0^{\G\mu}|_{C_0}^2\\
    &=\exp\left(-\int_0^t \beta_s\;ds\right) W_2^2(\pi_0^\mu,\pi_0^{\G\mu};C_0).
\end{align*}
It remains to verify $\beta_t=\beta_c(C_t)$. To this end, note that
\begin{align*}
    &P_t^{-1/2}(\sF-K_c(C_t)\sH)^T P_t^{1/2} + P_t^{1/2}(\sF-K_c(C_t)\sH)P_t^{-1/2}+ P_t^{-1/2} R(P_t) P_t^{-1/2}    \\
    &=C_t^{1/2}(\sF-K_c(C_t)\sH)^T C_t^{-1/2} + C_t^{-1/2}(\sF-K_c(C_t)\sH)C_t^{1/2} \\
    &- C_t^{1/2}(C_t^{-1} F + F^T C_t^{-1} + C_t^{-1}\sans{\Sigma} C_t^{-1}-\sH^T\sG^{-1}\sH)C_t^{1/2}\\
    &=-C_t^{1/2}(C_t \sH^T\sG^{-1}\sH)^T C_t^{-1/2} - C_t^{-1/2}(C_t\sH^T\sG^{-1}\sH)C_t^{1/2}-C_t^{-1/2}\sSig C_t^{-1/2} + C_t^{1/2}\sH^T\sG^{-1}\sH C_t^{1/2}\\
    &=-C_t^{1/2}\sH^T\sG^{-1}\sH C_t^{1/2}-C_t^{-1/2}\sSig C_t^{-1/2}.
\end{align*}
Of course, the matrix on the right-hand side is strictly negative definite, thus $\beta_t=\beta_c(C_t)>0$, $\beta_\infty:=\lim_{t\rightarrow\infty}\beta_t>0$, and $\beta_t\geq \beta_0>0$ for some $\beta_0$ that may depend on $\Cov(\mu)$.
\end{proof}

\begin{remark}[Consistency of $\beta_c$ with $\beta_d$]\label{rmk:beta_cd}
We note that $\beta_c$ can be obtained as the limit of $\beta_d$ under appropriate scaling of the matrices involved. Indeed, setting \eqref{def:tau_scale}, we see
\begin{align*}
    \beta_d(C)&=\lambda_{\min}\left( C^{1/2}H^T(H C H^T+\Gamma)^{-1}H  C^{1/2}+ C^{-1/2}( C^{-1}+H^T\Gamma^{-1}H+A^T\Sigma^{-1}A)^{-1} C^{-1/2}\right)\\
    &=\lambda_{\min}\left(\tau^2 C^{1/2}\sH^T(\tau^2 \sH C\sH^T+\tau\sG)^{-1}\sH C^{1/2}\right.\\
    &\left.\qquad\qquad+C^{-1/2}[C^{-1}+\tau \sH^T\sG^{-1}\sH+\tau^{-1}(I+\tau\sF)^T\sSig^{-1}(I+\tau\sF)]^{-1}C^{-1/2}\right)\\
    &=\tau\lambda_{\min}(C^{1/2}\sH^T\sG^{-1}\sH C^{1/2}+C^{-1/2}\sSig C^{-1/2}+O(\tau))\approx \tau \beta_c(C).
\end{align*}
Thus, for $\tau\ll 1$, we have
\[(1-\beta_d(C))\approx (1-\tau\beta_c(C))\approx e^{-\tau\beta_c(C)}.\]
\end{remark}

\begin{remark}\label{rmk:beta_c;interpretation}
    Note that in continuous-time, the dependence of $\beta_c$ on the size of the covariance is clearer. Indeed, under the rescaling $C\mapsto a C$ by some $a>0$, the first component $C^{1/2}\sH^T\sG^{-1}\sH C^{1/2}$ scales like $a$, which means that observation leads to a stronger contraction when uncertainty is high. On the other hand, $C^{-1/2}\sSig C^{-1/2}$ scales like $a^{-1}$, which can be understood as the signal noise destroying information faster when uncertainty is lower.
\end{remark}

In the following remark we compare Theorem~\ref{thm:gauss_lyapunov;cts} with the asymptotic stability results on the continuous-time MFEnKF due to Del Moral and Tugaut~\cite{DelMoralTugaut18stability}.

\begin{remark}\label{rmk:DelMoral}
 Let $(X_t^\mu)_{t\geq 0},(X_t^\nu)_{t\geq 0}$ be two solutions of~\eqref{eq:EnKF;mf;cts} with initial data $\mu,\nu\in\P_2(\R^d)$. In \cite[Theorem 3.4]{DelMoralTugaut18stability}, Del Moral and Tugaut show that the Wasserstein distance between laws of $X_t^\mu$ and $X_t^\nu$ (rather than their conditional laws given the observations) exponentially decays in time. While this result allows for possibly non-Gaussian $\mu,\nu$, the randomness in observations are averaged out, hence it differs from our result which holds pathwise almost surely. On the other hand, their theorem on pathwise almost sure exponential convergence rate measured in relative entropy, \cite[Theorem 3.5]{DelMoralTugaut18stability}, assumes Gaussian initial data.

In addition, both of their results are asymptotic in nature; the convergence rates have multiplicative constant depending on initial data and an additional parameter. This is in contrast with our Theorem~\ref{thm:gauss_lyapunov;cts}, which focuses on strictly contractive bounds at the cost of using a covariance-dependent norm; we will see in Theorem~\ref{thm:MFEnKF;stable;cts} that the exponential convergence in the usual Wasserstein distance follows.

Moreover, the exponential rates in~\cite{DelMoralTugaut18stability} are characterized in terms of the logarithmic norm of $\sF-K_c(C_\infty)\sH$; the negativity of the logarithmic norm associated to the Euclidean norm is a stronger requirement than the negativity of the spectral abscissa. If one is willing to work with any logarithmic norm, this gap is less serious. As discussed in~\cite[Section 8.1]{BishopDelMoral23survey}, $\alpha(\sF-K_c(C_\infty)\sH)<0$ (where $C_\infty$ is the limiting covariance matrix) implies by~\cite[Theorem 5]{Strom75logarithmic} that there exists a particular norm (depending on $C_\infty$) such that the corresponding logarithmic norm of $\sF-K_c(C_\infty)\sH$ is strictly negative. In fact, our strictly contractive bounds imply that the norm weighted by the limiting covariance matrix is an explicit example of such a norm, whereas by using the covariance-weighted (hence time-dependent) norms we are able to obtain strict contraction at each time, which is impossible in general if one has to fix the norm.
\end{remark}

We now turn to our contribution {\bf (C3)}, namely the stability and accuracy of the MFEnKF. We first begin by establishing the asymptotic stability with respect to the classical 2-Wasserstein distance.

\begin{theorem}\label{thm:MFEnKF;stable;cts}
Suppose $\sF,\sSig\in\R^{d\times d}$, $\sH\in\R^{k\times d}$, $\sG\in\R^{k\times k}$ satisfy Assumption~\ref{ass:cts}.
Fix $\mu,\nu\in\P_2(\R^d)$ with $\Cov(\mu),\Cov(\nu)\succ 0$, and let $\pi_t^\mu,\pi_t^\nu$ be the MFEnKF distributions associated to \eqref{eq:EnKF;mf;cts} with respective initial data $\mu,\nu$. Then, for almost every sample path of $(Z_t^\dag)_{t\geq 0}$,
\begin{equation}\label{eq:MFEnKF;cts;stable}
\begin{split}
    \limsup_{t\rightarrow\infty}\frac1t \log W_2(\pi_t^\mu,\pi_t^\nu)&\leq \alpha(\sF-K_c(C_\infty)\sH)<0.
\end{split}
\end{equation}
\end{theorem}

\begin{proof}
Denote $\lambda_\infty:=-\alpha(\sF-K_c(C_\infty)\sH)$, which is strictly positive, by \eqref{eq:alpha<0}, Proposition~\ref{prop:CARE}.
By the triangle inequality
\begin{align*}
    W_2(\pi_t^\mu,\pi_t^\nu)\leq W_2(\pi_t^\mu,\pi_t^{\G\mu})+ W_2(\pi_t^{\G\mu},\pi_t^{\G\nu})+W_2(\pi_t^{\nu},\pi_t^{\G\nu}).
\end{align*}
The first and third terms on the right-hand side can be taken care of by a minor modification of the coupling argument in Theorem~\ref{thm:gauss_lyapunov;cts}, where, conditioned on the common observation path, the coupling $\gamma_t$ between $\pi_t^\mu,\pi_t^{\G\mu}$ uses the $W_2$-optimal coupling of the initial laws $\mu,\G\mu$ instead.
Then, under $\gamma_t$,
\[\frac{d}{dt}(X_t^\mu-X_t^{\G\mu})=  (\sF-K_c(C_t)\sH) (X_t^\mu-X_t^{\G\mu}) \text{ for all } t>0,\]
where $K_c$ is the continuous-time Kalman gain defined in \eqref{def:Kc}. As $C_t\rightarrow C_\infty$ with $\lambda_\infty=-\alpha(\sF-K_c(C_\infty)\sH)>0$, for each $\eps>0$, we can fix a deterministic constant $c$ depending on $\eps,\mu$ but independent of $X_0^\mu, X_0^{G\mu}$ such that
\[|X_t^\mu-X_t^{\G\mu}|\leq c e^{-(\lambda_\infty-\eps)t}|X_0^\mu-X_0^{\G\mu}|;\]
squaring, integrating with respect to $\gamma_t$, then taking a square root, we obtain
\[W_2(\pi_t^\mu,\pi_t^{\G\mu})\leq c e^{-(\lambda_\infty-\eps)t} W_2(\mu,\G\mu).\]
By replacing $\mu$ with $\nu$, we have the same bound for $W_2(\pi_t^\nu,\pi_t^{\G\nu})$, and thus
\begin{equation}\label{eq:limsup;W2Gproj}
    \limsup_{t\rightarrow\infty}\frac1t \log\left[W_2(\pi_t^\mu,\pi_t^{\G\mu})+W_2(\pi_t^{\nu},\pi_t^{\G\nu})\right]\leq -\lambda_\infty<0.
\end{equation}

Thus it remains to show that, almost surely
\[\limsup_{t\rightarrow\infty}\frac1t \log W_2(\pi_t^{\G\mu},\pi_t^{\G\nu})\leq -\lambda_\infty<0.\]
As $\pi_t^{\G\mu}=\bar\pi_t^{\G\mu}$ and $\pi_t^{\G\nu}=\bar\pi_t^{\G\nu}$, this is a consequence of the results of Ocone and Pardoux~\cite{OconePardeoux96asymptotic} on the stability of the optimal filter with Gaussian initial data. We sketch the argument below for completeness. As $\bar\pi_t^{\G\mu},\bar\pi_t^{\G\nu}$ are Gaussians, writing $C_t^\mu=\Cov(\pi_t^\mu)$ and $C_t^\nu=\Cov(\pi_t^\nu)$, and using the explicit formula for the Wasserstein distance between Gaussians,
\[W_2^2(\bar\pi_t^{\G\mu},\bar\pi_t^{\G\nu})
= |\bar\pi_t^{\G\mu}(\id)-\bar\pi_t^{\G\nu}(\id)|^2+\tr\left(C_t^\mu+ C_t^\nu-2[(C_t^\mu)^{1/2}C_t^\nu (C_t^\mu)^{1/2}]^{1/2}\right).\]
By the standard property of the trace, $\tr[((C_t^\mu)^{1/2}C_t^\nu (C_t^\mu)^{1/2})^{1/2}]\geq \tr((C_t^\mu)^{1/2}(C_t^\nu)^{1/2})$, and thus
\begin{align*}
    \tr\left(C_t^\mu+ C_t^\nu-2[(C_t^\mu)^{1/2}C_t^\nu (C_t^\mu)^{1/2}]^{1/2}\right)&\leq \tr(C_t^\mu)+\tr (C_t^\nu)-2\tr((C_t^\mu)^{1/2}(C_t^\nu)^{1/2})\\
    &=\|(C_t^\mu)^{1/2}-(C_t^\nu)^{1/2}\|_F^{2}\leq \frac{\|C_t^\mu-C_t^\nu\|_F^2}{\left(\sqrt{\lambda_{\min}(C_t^\mu)}+ \sqrt{\lambda_{\min}(C_t^\nu)}\right)^2}
\end{align*}
where $\|\cdot\|_F$ is the Frobenius norm, and we applied~\cite[Theorem 6.2]{Higham08} (see \cite[Proposition 3.2]{vanHemAndo80} for a more general result) to obtain the last inequality.
As $0\prec C_t^\mu,C_t^\nu\rightarrow C_\infty$, the constant $\sqrt{\lambda_{\min}(C_t^\mu)}+\sqrt{\lambda_{\min}(C_t^\nu)}$ can be made independent of $t>0$.

Moreover, Theorem 2.3 of~\cite{OconePardeoux96asymptotic} states that 
\[\limsup_{t\rightarrow\infty}\frac1t\log|\bar\pi_t^{\G\mu}(\id)-\bar\pi_t^{\G\nu}(\id)|\leq -\lambda_\infty \text{ almost surely. }\]
The same upper bound holds for $\norm{C_t^\mu-C_t^\nu}_F$ by \eqref{eq:CARE;rate}, hence for $W_2(\pi_t^{\G\mu},\pi_t^{\G\nu})$. Combining with \eqref{eq:limsup;W2Gproj}, we conclude the proof.
\end{proof}

Now we turn to the question of accuracy of the MFEnKF. 
Combining with a recent result by V.H. Nguyen and the second named author~\cite[Corollary 4.9]{NguyenPark2026} we deduce that the MFEnKF distribution converges asymptotically exponentially to the optimal filter distribution in large time.

\begin{corollary}\label{cor:MFEnKF_accurate;cts}
Suppose $\sF,\sSig\in\R^{d\times d}$, $\sH\in\R^{k\times d}$, $\sG\in\R^{k\times k}$ satisfy Assumption~\ref{ass:cts}. Denote by $\mu,\nu\in\P(\R^d)$ the respective initial laws for the MFEnKF distribution $(\pi_t^\mu)_{t\geq 0}$ and the optimal filtering distribution $(\bar\pi_t^\nu)_{t\geq 0}$.
Let the initial laws satisfy the following conditions:
\begin{listi}
    \item $\mu\in\P_2(\R^d)$ with $\Cov(\mu)\succ 0$;
    \item $\nu\in\P(\R^d)$ satisfies the tail condition
    \begin{equation}\label{eq:nu;tailcond}
    \int_{\R^d}e^{-\frac12 x^T ((\theta Q_\infty)^{-1}-C_\infty^{-1}) x} d\nu(x)<+\infty \text{ for some } \theta>1
    \end{equation}
    where 
    \begin{equation}\label{eq:Qinfty}
        Q_\infty:=\int_0^\infty e^{-(\sF+\sSig C_\infty^{-1})s}\sSig e^{-(\sF+\sSig C_\infty^{-1})^T s}\;ds.
    \end{equation}
\end{listi}

Then, for almost every sample path of $(Z_t^\dag)_{t\geq 0}$,
    \begin{equation}\label{eq:MFEnKF_accurate;cts}
        \limsup_{t\rightarrow\infty}\frac1t \log W_2(\pi_t^\mu,\bar\pi_t^\nu)\leq \alpha(\sF-K_c(C_\infty)\sH)<0 \text{ as } t\rightarrow\infty.
    \end{equation}
\end{corollary}

\begin{proof}
    Recall that $\pi_t^{\G\mu}=\bar\pi_t^{\G\mu}$. Thus, by the triangle inequality,
    \begin{align*}
        W_2(\pi_t^\mu,\bar\pi_t^\nu)\leq W_2(\pi_t^\mu,\pi_t^{\G\mu})+W_2(\bar\pi_t^{\G\mu},\bar\pi_t^{\nu}).
    \end{align*}
    Asymptotic exponential decay of the first term on the right-hand side has been established in Theorem~\ref{thm:MFEnKF;stable;cts}. 
    
    It remains to show that the second term decays asymptotically exponentially at the same rate. This is a direct consequence of~\cite[Corollary 4.9]{NguyenPark2026}, which states that whenever $\nu,\nu'\in\P(\R^d)$ satisfy the condition \eqref{eq:nu;tailcond}, then
    \[\limsup_{t\rightarrow\infty}\frac1t\log W_2(\bar\pi_t^{\nu'},\bar\pi_t^{\nu})\leq \alpha(\sF-K_c(C_\infty)\sH)<0\]
    for almost every sample path of $(Z_t^\dag)_{t\geq 0}$. 

    As $\nu$ satisfies the tail condition by assumption, the desired result \eqref{eq:MFEnKF_accurate;cts} would follow once we verify that $\G\mu$ also satisfies the tail condition, by applying the result to $\nu$ and $\nu'=\G\mu$.
    The tail condition holds for any nondegenerate Gaussian. Indeed, as $Q_\infty\preceq C_\infty$ by \cite[Lemma 2.8, equation (2.25)]{NguyenPark2026}, it follows that
    \[Q_\infty^{-1}-C_\infty^{-1}+\Cov(\G\mu)^{-1}\succeq \Cov(\mu)^{-1}\succ 0.\]
    Thus, for sufficiently small $\theta>1$ we have
    \begin{align*}
        \int_{\R^d} \exp\left(-\frac12 x^T((\theta Q_\infty)^{-1}-C_\infty^{-1}+\Cov(\G\mu)^{-1})x\right)\;dx<+\infty,
    \end{align*}
    which implies the tail condition for $\G\mu$, as the mean of $\G\mu$ contributes linearly in the exponent, which does not affect the integrability. This completes the proof.
\end{proof}

\begin{remark}\label{rmk:tail_cond}
    While~\cite{NguyenPark2026} considers the case $\sSig=I_d$ to simplify notation, all the results generalize to the case of nondegenerate $\sSig\succ0$ by coordinate change. We have accordingly modified the statement (such as the exponent and $Q_\infty$). Also, while the exponent in the cited theorem is characterized in terms of $-\sF-\sSig C_\infty^{-1}$, this matrix is similar to $(\sF-K_c(C_\infty)\sH)^T$ hence their spectral abscissa coincide.

    Also, the tail condition in the said paper involves translation by some fixed $q_0$ in \eqref{eq:nu;tailcond} -- i.e.
    \[\int_{\R^d}e^{-\frac12 (x-q_0)^T ((\theta Q_\infty)^{-1}-C_\infty^{-1}) (x-q_0)} d\nu(x)<+\infty \text{ for some } \theta>1.\]
    One can readily check that the translation is inconsequential for this condition, as $\theta$ can be replaced by a slightly smaller number, and this change does not affect the asymptotic rate of decay.
\end{remark}

\section{Conclusions}

In this paper, we have precisely characterized the large-time behavior of the MFEnKF in the linear setting with Gaussian noise where the signal-observation pair is detectable and the signal noise is nondegenerate, by showing nonasymptotic exponential contraction towards the Gaussian manifold. This was made possible by identifying the variational structure of the algorithm, which involves the covariance-weighted Wasserstein distance. We note here that a similar variational argument applies in the case of nonlinear observation -- i.e. when $x\mapsto Hx$ is replaced by $x\mapsto h(x)$. In this case, the resulting scheme does not coincide with the MFEnKF, as the corresponding Euler-Lagrange equation involves the gradient of $h$; thus we did not treat the more general case in this paper.

Moreover, in the same setting, we showed that the MFEnKF accurately approximates the optimal filter in large time regardless of the choice of initial data for the two filters, which may be unexpected given that each recursive step may poorly approximate the optimal filtering algorithm. This reminds us that an effective filtering algorithm need not approximate the Bayes update well at each step, and what is perhaps more important is to emulate the large-time behavior of the optimal filter. For these results, we have heavily relied on the theory of matrix Riccati equations that characterize the limiting shape of the optimal filter in the linear setting with Gaussian noise. We believe that characterizing the large-time behavior of the optimal filter in the nonlinear setting, where we expect a richer and more complex behavior, is not only of mathematical interest but also important for designing new filtering algorithms.

Finally, it is natural to ask if improved global-in-time mean-field convergence results for the EnKF could be obtained using the nonasymptotic quantitative stability estimates in Theorems~\ref{thm:gauss_lyapunov;disc} and \ref{thm:gauss_lyapunov;cts}.

\vspace{0.75in}

\noindent{\bf Acknowledgments}
FH is supported by NSF CAREER Award 2340762 and by Caltech start-up funds.
AMS is supported by the Department of Defense (DoD) Vannevar Bush Faculty Fellowship, award N00014-22-1-2790; SP is also supported through this award and by the von Karman Instructorship in the Computing + Mathematical Sciences Department at Caltech.

\medskip

\noindent{\bf Statement of AI Use.}
In preparing the first version of the manuscript, ChatGPT-5.4 and ChatGPT-5.5, accessed by SP through ChatGPT Plus, were used in a limited capacity for literature searches and for performing and cross-verifying routine matrix calculations.

In preparing the second version of this manuscript, ChatGPT-5.6 Sol, accessed by SP through ChatGPT Plus, was used in developing a detailed argument for extending Gelfand's formula to products of a convergent sequence of matrices in the proof of Proposition~\ref{prop:gauss_conv;W2;asymp}. In reviewing the second version of the manuscript, ChatGPT-6 Pro, provided to SP through OpenAI's ChatGPT for Academic Researchers Program, was used to provide feedback on a complete draft. This process revealed various typos and minor errors or gaps in the proofs. One nontrivial suggestion resulting from this process was incorporated into the proof of Lemma~\ref{lem:entropy;cvx}. 

The authors have carefully verified, digested, and reorganized all suggestions made by the AI models.

We emphasize that AI models were not used to produce any of the main ideas contained in this manuscript. Moreover, they were not used to generate any prose in the manuscript, and the authors assume full responsibility for all content of this paper.
\nc
\bibliography{bib}
\bibliographystyle{siam}

\begin{appendix}

\section{Preliminaries on the matrix-weighted Wasserstein distances}\label{app:W2M}

In this section, we provide the proofs of Proposition~\ref{prop:funct_ineq;W2M} and Lemma~\ref{lem:EL;potential}; we will first establish a few preliminary lemmas. The arguments used are straightforward modifications of those in the book of Ambrosio, Gigli, and Savar\'{e}~\cite{AGS}. Hence, our proofs will be often brief, pointing to the corresponding arguments in the book and stating necessary modifications.

\begin{lemma}\label{lem:potential;cvx}
    Let $U\in \C^2(\R^d)$, and suppose there exists $M\in\R^{d\times d}_{++}$ such that $D^2  U(x)\succeq M^{-1}$ for all $x\in\R^d$. Then the potential energy
    \begin{equation}\label{def:V}
        \V(\mu):=\int_{\R^d}U(x)\;d\mu(x)
    \end{equation}
    is 1-convex along $W_{2}(\cdot,\cdot;M)$-geodesics -- i.e. 
    \begin{equation}\label{eq:V;W2M_geocvx}
        \V(\mu_t)\leq (1-t)\V(\mu)+t\V(\nu)-\frac{t(1-t)}{2}W_{2}^2(\mu,\nu;M)
    \end{equation}
    for any constant-speed $W_2(\cdot,\cdot;M)$-geodesic $(\mu_t)_{t\in[0,1]}$ from $\mu\in\P_2(\R^d)$ to $\nu\in\P_2(\R^d)$.
\end{lemma}
\begin{proof}
    Fix any $x,y\in\R^d$, and define $x_t:=(1-t)x+t y$. Then by standard calculations involving the first two derivatives of $f(t):= U(x_t)$, most importantly that 
    \[f''(t)=(y-x)^TD^2 U(x_t) (y-x)\geq |x-y|_M^2,\]
    we can deduce
    \[U(x_t)\leq (1-t)U(x)+t U(y)-\frac{t(1-t)}{2}|x-y|_M^2.\]
    We then conclude \eqref{eq:V;W2M_geocvx} by integrating over an optimal coupling $\gamma(dxdy)\in\Pi_{o}(\mu,\nu;M)$.
\end{proof}

We can also deduce that the entropy functional is convex along $W_{2}(\cdot,\cdot;M)$-geodesics.
\begin{lemma}\label{lem:entropy;cvx}
    The entropy functional $\Ent:\P_2(\R^d)\rightarrow(-\infty,+\infty]$ defined by
    \begin{equation}\label{def:Ent}
    \Ent(\mu):=
    \begin{cases}
        \int_{\R^d}\log\frac{d\mu}{d\Leb^d}\;d\mu &\text{ if } \mu\ll\Leb^d,\\
        +\infty &\text{ otherwise }
    \end{cases}
    \end{equation}
    is convex along $W_2(\cdot,\cdot;M)$-geodesics for any $M\in\R^{d\times d}_{++}$.
\end{lemma}

\begin{proof}
    We use the representation of the matrix-weighted Wasserstein distances as the Wasserstein distance between linearly transformed measures, studied by Carrillo and Vaes~\cite{CarVaes21EKI}. Let $L:=M^{-1/2}$, and we abuse notation to write $L_\#\sigma$ as the pushforward of $\sigma\in\P_2(\R^d)$ under the linear map $x\mapsto Lx$. Then, as $|x-y|_{M}^2=|L(x-y)|^2$,
    \[W_2(\mu,\nu;M) = W_2(L_\#\mu,L_\#\nu).\]
    Moreover, if $(\mu_t)_{t\in[0,1]}$ is a constant-speed $W_{2,M}$-geodesic from $\mu$ to $\nu$, then $\tilde\mu_t:= L_\#\mu_t$ for $t\in[0,1]$ is a constant-speed $W_2$-geodesic from $L_\#\mu$ to $L_\#\nu$. Thus, by the geodesic convexity of $\Ent$ with respect to $W_2$-geodesics,
    \begin{equation}\label{eq:Ent;tildemu}
        \Ent(L_\#\mu_t)=\Ent(\tilde\mu_t)\leq (1-t)\Ent(\tilde\mu_0) + t\Ent(\tilde\mu_0) = (1-t)\Ent(L_\#\mu)+t\Ent(L_\#\nu).
    \end{equation}
    On the other hand, for any measure $\sigma\in\P_2(\R^d)$,
    \[\Ent(L_\#\sigma)=\Ent(\sigma)+\frac12\log\det M.\]
    Thus, substituting into \eqref{eq:Ent;tildemu}, we deduce that
    \[\Ent(\mu_t)\leq (1-t)\Ent(\mu)+t(\nu),\]
    which is precisely the convexity of $\Ent$ along $W_{2,M}$-geodesics.
\end{proof}
We note that in fact, $\Ent$ satisfies a stronger property, namely convexity along generalized geodesics with respect to $W_2(\cdot,\cdot;M)$.
\medskip

We now present the proof of Proposition~\ref{prop:funct_ineq;W2M}.

\begin{proof}[Proof of Proposition~\ref{prop:funct_ineq;W2M}]
    For simplicity, we write $\Ecal(\mu):=\kl(\mu||\pi)$ in this proof. The first statement \eqref{eq:KL;cvx} follows from adding the analogous inequalities satisfied by the potential energy and the entropy, which have been established in Lemmas~\ref{lem:potential;cvx}-\ref{lem:entropy;cvx}. 

    To show~\eqref{eq:Talagrand;W2M}, fix any $\mu\in\P_2(\R^d)$ and let $(\mu_t)_{t\in[0,1]}$ be the constant speed geodesic between $\mu_0=\pi$ and $\mu_1=\mu$. Then reorganizing the terms in ~\eqref{eq:KL;cvx}, we have, for all $t\in(0,1)$,
    \[\frac{\Ecal(\mu_t)-\Ecal(\pi)}{t}\leq \Ecal(\mu)-\Ecal(\pi)-\frac{1-t}{2}W_2^2(\mu,\pi;M).\]
    Recalling that $\pi$ is the minimizer of $\Ecal$ we see the left-hand side is nonnegative, thus we deduce \eqref{eq:Talagrand;W2M} by reorganizing, noting $\Ecal(\pi)=0$, and taking $t\rightarrow 0$.

    The log-Sobolev-type inequality \eqref{eq:LSI;W2M} follows from \eqref{eq:HWI;W2M} by 
    completing the square. 
    Thus it remains to show the HWI-type inequality~\eqref{eq:HWI;W2M}. It is a direct consequence of two facts about the (local) metric slope of $\Ecal$, defined by
    \begin{equation}\label{eq:localslope}
        |\partial\Ecal|_{M}(\mu):=\limsup\left\{\frac{(\Ecal(\mu)-\Ecal(\nu))_+}{W_{2}(\mu,\nu;M)}:\;\nu\rightarrow \mu \text{ w.r.t. } W_{2,M}\right\},
    \end{equation}
    where $(a)_+:=\max\{a,0\}$. Firstly, due to the 1-geodesic convexity of $\Ecal$, the local metric slope allows the representation~\cite[Theorem 2.4.9]{AGS}
    \[|\partial\Ecal|_M(\mu)=\sup_{\nu\neq \mu}\left\{\frac{\Ecal(\mu)-\Ecal(\nu)}{W_{2,M}(\mu,\nu)}+\frac12 W_{2,M}(\mu,\nu)\right\},\]
    reorganization of which leads to
    \[\Ecal(\mu)\leq W_{2,M}(\mu,\pi)|\partial\Ecal|_M(\mu)-\frac{1}{2}W_{2,M}^2(\mu,\pi).\]
    Secondly one can verify that $|\partial\Ecal|_M^2(\mu)=I_M(\mu||\pi)$, which is implied for instance by a much more general result due to Ambrosio, Gigli, and Savar\'{e}~\cite[Theorem 9.3]{AGS14Inventiones}. 
\end{proof}

We conclude with a proof of Lemma~\ref{lem:EL;potential}.
\begin{proof}[Proof of Lemma~\ref{lem:EL;potential}]
    The proof is a straightforward adaptation of arguments due to Ambrosio, Gigli, and Savar\'{e}~\cite[Lemma 10.3.4, Proposition 10.4.2]{AGS}, but we include a proof here for completeness. As $V$ is bounded from below, the variational problem~\eqref{eq:JKO;V} has a minimizer; fix a minimizer $\pi^\ast\in\P_2(\R^d)$. We claim that it suffices to show that for any $\gamma\in\Pi_{o}(\pi^\ast,\pi;M)$, $T\in L^2(\gamma;\R^d)$ and $\nu=T_\#\gamma\in\P_2(\R^d)$, 
    \begin{equation}\label{eq:strong_subdiff}
    \begin{split}
        \int V(y)\;d\nu(y) - \int V(x)\;d\pi^\ast(x)\geq &\langle M^{-1}(y-x),T(x,y)-x\rangle_{L^2(\gamma;\R^d)}\\
        &-\frac{1}{2}\|M^{-1/2}(T-\id^1)\|_{L^2(\gamma;\R^d)}^2.
    \end{split}
    \end{equation}
    Indeed, suppose \eqref{eq:strong_subdiff} is true. Choose $T(x,y)-x=\eps v(x,y)$ for some vector field $v\in \C_c^\infty(\R^{2d};\R^d)$. Then
    \begin{align*}
    -\frac{1}{2} \eps^2 \norm{M^{-1/2}v}_{L^2(\gamma;\R^d)}^2 &+\eps\langle M^{-1} (y-x), v\rangle_{L^2(\gamma;\R^d)}  
    \leq \int V(y)\;d\nu(y)-\int V(x)\;d\pi^\ast(x)  \\
    &=\int (V(T(x,y))-V(x))\;d\gamma(x,y)= \int (V(x+\eps v(x,y))-V(x))\;d\gamma(x,y)\\
    &= \langle \eps\nabla V(x),v(x,y)\rangle_{L^2(\gamma;\R^d)}+o(\eps).
    \end{align*}
    Dividing both sides by $\eps$ and letting $\eps\to 0$ we see
    \[\langle M^{-1}(y-x),v(x,y)\rangle_{L^2(\gamma;\R^d)}\leq \langle \nabla V(x),v(x,y)\rangle_{L^2(\gamma;\R^d)}.\]
    Replacing $v$ with $-v$ we see that in fact the above inequality is an equality for any $v\in \C_c^\infty(\R^{2d};\R^d)$. Hence $y-x=M\nabla V(x)$ for $\gamma$-a.e $(x,y)\in\R^d\times\R^d$. In particular, this means $\nabla V\in L^2(\gamma)$, and we may conclude that any $W_{2,M}$-optimal transport plan $\gamma$ is given by~\eqref{eq:EL;potential}.

    Thus it remains to show \eqref{eq:strong_subdiff}. Fix any $\gamma\in\Pi_{o}(\pi^\ast,\pi;M)$ and $T\in L^2(\gamma;\R^d)$, and define $\nu=T_\#\gamma \in\P_2(\R^d)$, noting that $(T\times \id^2)_\#\gamma\in\Pi(\nu,\pi)$, where $\id^2$ is the projection onto the second $d$-dimensional coordinate. By the definition of $\pi^\ast$,
    \begin{align*}
        \int V(y)\;d\nu(y)&-\int V(x)\;d\pi^\ast(x)
        \geq \frac{1}{2}\left(W_{2,M}^2(\pi^\ast,\pi)-W_{2,M}^2(\nu,\pi)\right)\\
        &\geq \frac{1}{2}\int \left(|y-x|_M^2-|y-T(x,y)|_M^2\right)\;d\gamma(x,y)  \\
        &=\int  \langle M^{-1}(y-x),T(x,y)-x\rangle\;d\gamma(x,y) - \frac{1}{2}\|M^{-1/2}(T-\id^1)\|_{L^2(\gamma;\R^d)}^2
    \end{align*}
In the last equality we used the elementary identity
\begin{align*}
    \frac12 |a|^2_M-\frac12 |b|_M^2 = \langle a,a-b\rangle_M - \frac12 |a-b|_M^2.
\end{align*}
\end{proof}

\end{appendix}

\end{document}